\documentclass[11pt]{amsart}
\usepackage{palatino}
\usepackage{amsmath}
\usepackage{amssymb, graphicx, epsfig}

\usepackage[margin=1.5cm,
 vmargin={0pt,1cm},
  nohead]{geometry}
\newtheorem{theorem}{$\quad$Theorem}[section]

\newtheorem{corollary}[theorem]{$\quad$Corollary}

\newtheorem{lemma}[theorem]{$\quad$Lemma}

\newcommand{\R}{{\mathbb R}}
\newcommand{\C}{{\mathbb C}}
\newcommand{\Z}{{\mathbb Z}}

\theoremstyle{definition}
\newtheorem{remark}[theorem]{$\quad$Remark}

\newcounter{bibno}

\begin{document}
\markboth{Johan Bj\"orklund}
{Real algebraic knots of low degree}
\title{Real algebraic knots of low degree}

\author{Johan Bj\"orklund}
\address{Department of mathematics, Uppsala University\\ Box 480, 751 06 Uppsala, Sweden\\ email:bjorklund@math.uu.se}
\date{October 14, 2010}

\begin{abstract}
In this paper we study rational real algebraic knots in $\R P^3$.  We show that two real rational algebraic knots of degree $\leq5$ are rigidly isotopic if and only if their degrees and encomplexed writhes are equal. We also show that any smooth irreducible knot which admits a plane projection with less than or equal to four crossings has a rational parametrization of degree $\leq 6$. Furthermore an explicit construction of rational knots of a given degree with arbitrary encomplexed writhe (subject to natural restrictions) is presented.

\end{abstract}

\maketitle



\clearpage

\section{Introduction}

Classical knot theory is the study of smooth embeddings of $S^1$ into $\R^3$ or $S^3$ up to smooth isotopy. In this paper we study embeddings of $S^1$ into $\R P^3$ realized by polynomial functions from $\R P^1$ into $\R P^3$.\\

Let $K\subset\R P^3\subset \C P^3$ be a smooth knot. We say that a function $P:\C P^1\to\C P^3$ is a {\em rational parametrization} of degree $d$ of $K$ if $P$ is a smooth embedding with coordinate functions which are real homogenous polynomials of degree $d$ and if $P(\R P^1)=K$. In this case we also call $K=P(\R P^1)$ a {\em rational knot} and call $P(\C P^1)=\C K$ its {\em complexification}. Since smooth functions can be approximated arbitrarily well by polynomials, any smooth isotopy class of knots in $\R P ^3$ can be represented by a rational knot. We define the {\it rational degree} of a knot $K$ as the smallest number $d$ such that $K$ is smoothly isotopic to a knot having a rational parametrization of degree $d$ and show the following result for knots admitting plane projections with few double points.

\begin{theorem} \label{main}
A smooth irreducible knot in $\R P^3$ which admits a plane projection with at most four transversal double points has rational degree at most 6.
\end{theorem}

Theorem \ref{main} is proved in Section 5. A knot is {\em irreducible} if it is not {\em reducible}, that is, a connected sum of two non-trivial knots  such that one of them is contained in a ball.\\

The space of rational parametrizations of real rational curves of degree $d$ forms an open subset of the real projective space $\R P^{4d+3}$. The subspace of parametrizations which are not embeddings forms a closed subspace of codimension one which we call the {\em discriminant}. The complement of the discriminant is the {\em space of rational knot parametrizations}.

Two rational knots are {\em rigidly isotopic} if there exists a continuous path $P_t,0\leq t\leq1$ in the space of rational knot parametrizations connecting them. This notion of rigid isotopy was introduced by Rokhlin in \cite{ROHK} for real algebraic plane projective curves.  Rigid isotopy is obviously more refined than smooth isotopy since it preserves degree. More interestingly, there are smoothly isotopic rational knots of the same degree which are not rigidly isotopic. Such knots were found by Viro in \cite{VIRO} and were distinguished by their encomplexed writhe. The encomplexed writhe of a real algebraic knot $K$ is a sum of signs over the double points in any generic plane projection. In particular, also solitary double points (i.e. double points with preimages in complex conjugate branches) contributes. In this paper we denote the encomplexed writhe by $w$. The following theorem gives a complete rigid isotopy classification of rational knots up to degree 5.
\clearpage

\begin{theorem}\label{rigidclass}
There are the following 15 rigid isotopy classes of rational knots of degree $d\leq5$.
\begin{enumerate}
\item For $d=1$ there is only one rigid isotopy class:
\begin{itemize}
 \item the line with $w=0$.
\end{itemize}
\item For $d=2$ there is only one rigid isotopy class:
\begin{itemize}
 \item the circle with $w=0$.
\end{itemize}

\item For $d=3$ there are two rigid isotopy classes:
\begin{itemize}
 \item a line with $w=1$ and
 \item a line with $w=-1$.
\end{itemize}

\item For $d=4$ there are five rigid isotopy classes:
\begin{itemize}
 \item a circle with $w=2$,
 \item a circle with $w=0$,
 \item a circle with $w=-2$,
 \item a twocrossing knot (see Fig. \ref{upto5}) with $w=4$ and
 \item a mirror image of a twocrossing knot with $w=-4$.
\end{itemize}

\item For $d=5$ there are seven rigid isotopy classes:
\begin{itemize}
 \item a line with $w=2$,
 \item a line with $w=0$,
 \item a line with $w=-2$,
 \item a long trefoil (see Fig. \ref{upto5}) with $w=4$,
 \item a mirror image of a long trefoil with $w=-4$,
 \item a projective $5_3-$knot (see Fig. \ref{upto5}) with $w=6$ and
 \item a mirror image of a projective $5_3-$knot with $w=-6$.
\end{itemize}
\end{enumerate}
In particular, degree and encomplexed writhe form a complete rigid isotopy invariant for $d\leq5$.
\end{theorem}

\begin{figure}
\includegraphics[scale=0.35]{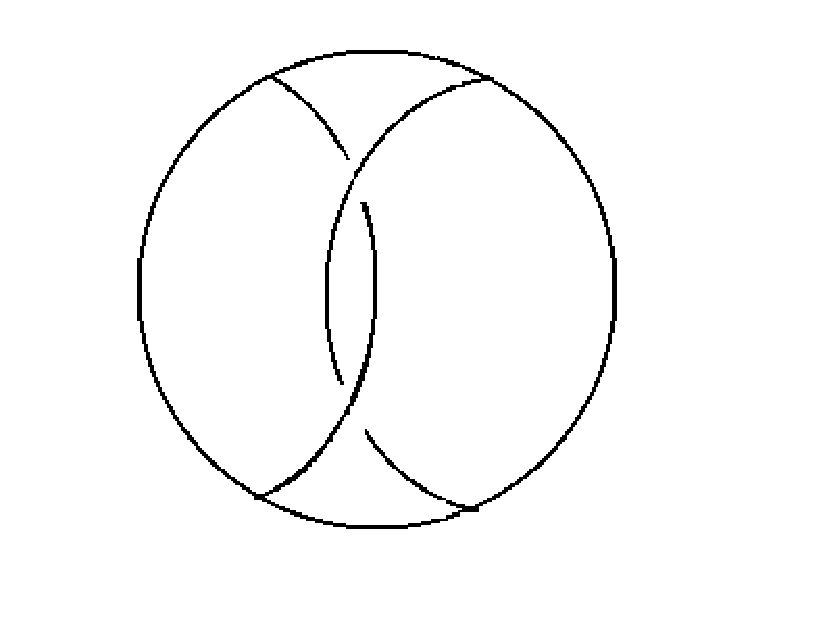}
\includegraphics[scale=0.30]{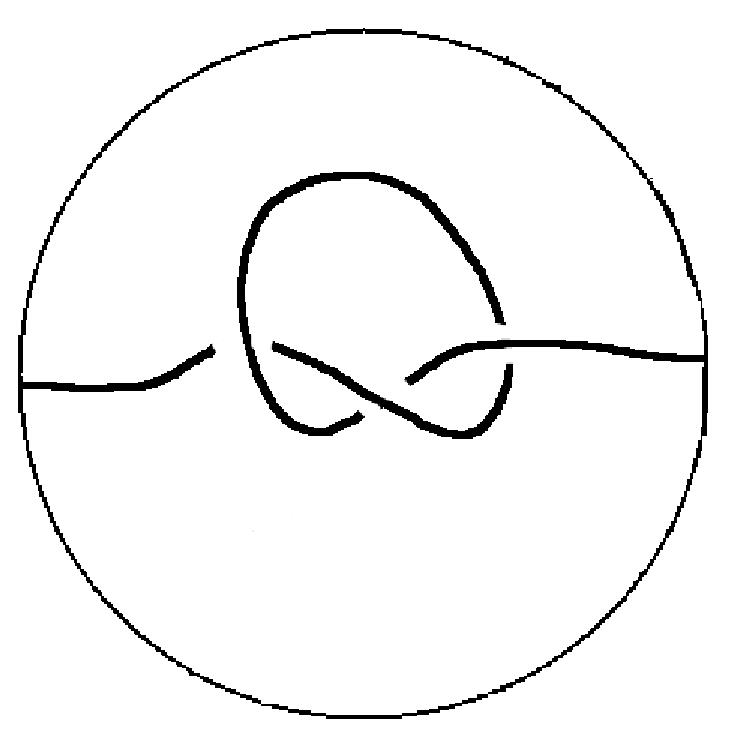}
\includegraphics[scale=0.30]{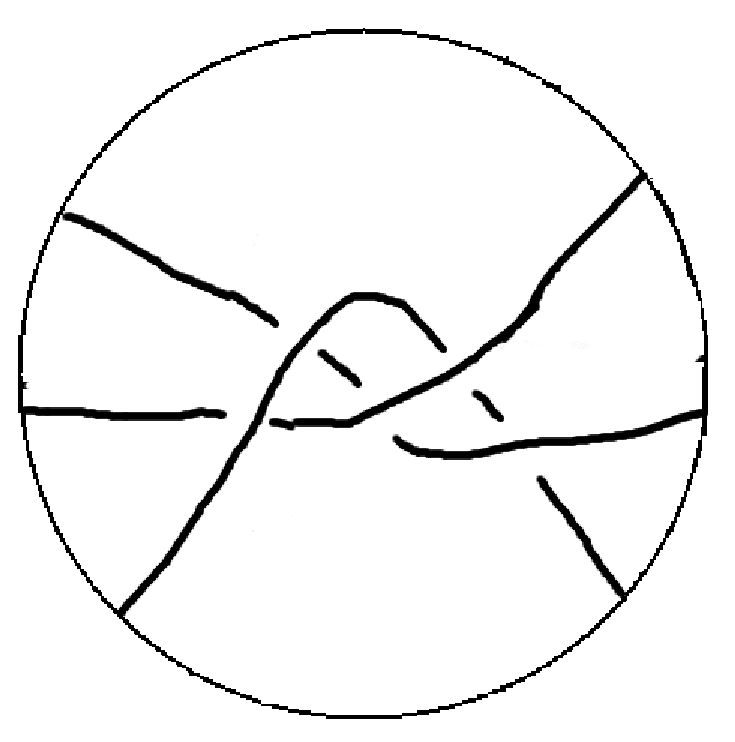}
\caption{Twocrossing knot, long trefoil and the projective $5_3-$knot}
\label{upto5}
\end{figure}

Theorem \ref{rigidclass} is proved in Section 4. The last statement of Theorem \ref{rigidclass} does not hold true for higher degrees, see Remark \ref{deg6counter} for a counterexample in degree 6.

We also consider the problem of realizing knots with different encomplexed writhes. It is straightforward to show that the encomplexed writhe $w$ of a rational curve of degree $d$ must satisfy $|w|\leq\frac{(d-1)(d-2)}{2}$ and $w\equiv\frac{(d-1)(d-2)}{2}\pmod{2}$. In \cite{VIRO} Viro states that the encomplexed writhe for a given degree attains all integer values subject to these restrictions. We give an explicit method for construction of rational degree $d$ knots $K_d^w$ of any writhe $w$ satisfying these conditions. More precisely we show the following result.

\begin{theorem}\label{allwrithe}
For each degree $d$ and writhe $w$ such that $|w|\leq\frac{(d-1)(d-2)}{2}$ and $w\equiv\frac{(d-1)(d-2)}{2}\pmod{2}$ the rational knot $K_d^w$ constructed as described in Section \ref{kdw} has degree $d$ and writhe $w$. Thus giving rational knots of each possible writhe for a given degree $d$.
\end{theorem}
The proof of the Theorem and the construction of $K_d^w$ is presented in Section 4.\\

The results of this paper are related to recent work of Mikhalkin and Orevkov\cite{ORMI} where all real algebraic links of degree at most $6$ were classified up to smooth isotopy.

\section{Preliminaries}
In this section we introduce notation for projective spaces and discuss some elementary properties of the rational degree of knots in $\R P^3$. Furthermore, the generalization of knot diagrams for knots in $\R^3$ to knots in $\R P^3$ is introduced and a theorem used for gluing knots is presented.

\subsection{Projective spaces}
The space $\R P^n$ is the set of lines through the origin in $\R^{n+1}$. A non-zero vector $x=(x_0,...,x_n)$ in $\R^{n+1}$ defines a point $p$ in $\R P^n$ corresponding to the line generated by $x$. The homogenous coordinates of $p$ are $(x_0\,{:}\, ...\,{:}\, x_n)$. Accordingly, two points $p_1=(x_0\,{:}\, ...\,{:}\, x_n)$, $p_2=(y_0\,{:}\, ...\,{:}\, y_n)$ coincide if and only if there exists a real number $\lambda\neq0$ such that $x_i=\lambda y_i$ for $i=0,...,n$. Analogously we define homogenous coordinates $(x_0\,{:}\, ...\,{:}\, x_n)$ of $\C P^n$ where $x_i$ denotes complex numbers. Mapping $(x_0\,{:}\, ...\,{:}\, x_n)\in\R P^n$ to the point with the same coordinates in $\C P^n$ gives a natural inclusion $\R P^n\subset \C P^n$. We call points in $\R P^n\subset\C P^n$ {\em real points} and points in $\C P^n\backslash\R P^n$ {\em non-real points}. Complex conjugation of $\C P^n$ sends $p=(x_0\,{:}\, ...\,{:}\, x_n)$ to $\bar{p}=(\bar{x}_0\,{:}\, ...\,{:}\, \bar{x}_n)$ where $\bar{x}_j$ denotes the complex conjugate of the complex number $x_j$. Throughout this paper we will denote homogenous coordinates in the case $n=1$ by $(t\,{:}\,s)$.\\

We often represent $\R P^n$ as the unit sphere $S^n$ in $\R^{n+1}$ with antipodal points identified. In particular,  mapping the unit disk $D^n$ to the upper half-sphere gives a parametrization $D^n\rightarrow\R P^n$ defined by formula $$(x_1,...,x_n)\mapsto(\sqrt{1-x_1^2-x_2^2-...-x_n^2}\,{:}\, x_1\,{:}\, x_2\,{:}\, ...{:}\,x_n)$$ where antipodal points on the boundary are identified. The image of the boundary (defined by homogenous equation $x_0=0$) is called the {\em plane at infinity} and points in this plane are called {\em points at infinity}.

\subsection{Elementary properties of the rational degree}

If $P$ is a real polynomial map of degree $d$ from $\C P^1$ to $\C P^3$ parametrizing a knot $K$ we obtain new parametrization of the same degree $d$ by composing $P$ on the left by real linear transformations of $\C P^3$ and on the right by real linear transformations of $\C P^1$. Real linear transformations of $\C P^n$ form the Lie group $PGL_{n+1}(\R)$. Acting with elements from $PGL_2(\R)$ will result in new parametrizations of $K$ while the action by $PGL_{4}(\R)$ may result in a different image. This allows us to consider the space of real rational knot parametrizations of a given degree $d$ as a fiber bundle with fiber $PGL_2(\R)$ and the space of rational knots as the base space.\\

Let $P$ be a degree $d$ parametrization of a knot $K$. Then $d$ is even if and only if $K$ realizes zero in $H_1(\R P^3;\Z)$. Each knot in $\R^3$ can be considered as an affine knot in $\R P^3$ using the natural inclusion $\R^3\subset\R P^3$ onto the complement of a hyperplane. Every knot in $\R P^3$ smoothly isotopic to an affine knot must then have even rational degree. The inverse is however not true, an example of a knot in $\R P^3$ that realizes zero in $H_1(\R P^3;\Z)$ but is not isotopic to a knot from $\R^3$ is the twocrossing knot depicted in Fig. \ref{upto5} (as the knot with exactly two crossings).

By examining intersections with algebraic surfaces, we obtain some limitations on possible knots of rational degree $d$. For very low degrees this gives an immediate classification. It is well known that a curve in $\C P^3$ of degree $d$ either intersects a surface of degree $n$ in at most $nd$ points (counted with multiplicity) or is contained in the surface. A rational knot of degree $d\leq2$ is then planar since any three distinct points on the curve lie in a plane. Such a knot of degree $d\leq 2$ is smoothly isotopic to the straight line if and only if it is of rational degree $d=1$ and smoothly isotopic to the circle if and only if it is of rational degree $d=2$.
A similar result for $d\leq4$ is as follows.

\begin{lemma}
A knot admitting a parametrization of degree $\leq4$ lies on a quadric.
\end{lemma}
\begin{proof}
Let $P$ be a rational parametrization with degree $d\leq4$. Choose nine distinct points $a_1,...,a_9$ contained in the knot. A surface of degree two is defined as the zero set of a homogenous polynomial $F$ of degree 2. Such a polynomial has ten coefficients. Since we have only nine points, a nontrivial collection of coefficients can be found such that all nine points are contained in the degree $2$ surface defined by the polynomial. This is only possible if the knot is contained in the surface.
\end{proof}

To get a limitation of the number of real crossings between a rational knot of degree $d$ and the plane at infinity  the following lemma is used.

\begin{lemma}\label{infcross}
Given a nonplanar knot $K$ parametrized by $P$ of degree $d>2$ there exists a rational parametrization $P'$ of degree $d$ parametrizing a knot $K'$ isotopic to $K$ such that $K'$ is transversal to the plane at infinity and intersects it in at most $d-2$ real points.
\end{lemma}

\begin{proof}
Take a real plane $H$ transversal to the knot $K$ intersecting its complexification $\C K$ in a nonreal point $p$. Since both the knot and the plane are defined by real equations they will also intersect each other at the complex conjugate point $\bar p\neq p$. The knot must intersect $H$ in exactly $d$ points and so at most $d-2$ of them can be real. A real linear transformation of $\C P^3$ that preserves orientation takes $H$ to the plane at infinity.\\
\end{proof}


\subsection{Knot diagrams in $\R P^3$}\label{knotdiagrams}

To study knots in $\R^3$, knot diagrams are often used to reduce a problem for curves in space to combinatorial manipulations of diagrams in a plane as follows. A given knot $K$ is projected to a generic plane $H$. In appropriate coordinates, the plane $H$ will be defined by $x_3=0$, and projecting to the plane will then correspond to disregarding the height-function $x_3$. Furthermore, since $H$ is generic, the projection is an immersion and any self intersection is a transversal double point. A {\em knot diagram of K} is the image of such a projection together with extra information at the double points: at each double point it is indicated which preimage is above the other, see Fig. \ref{cross}. The {\em diagram of $K$} is the image of the projection together with the extra information at the double points. Two knots are smoothly isotopic if and only if the diagram of one can be deformed into the diagram of the other by smooth planar isotopy and a series of moves called Reidemeister moves. These are depicted as $\Omega_1,\Omega_2$ and $\Omega_3$ in Fig. \ref{reidemeister}.\\

\begin{figure}
\includegraphics[scale=0.5]{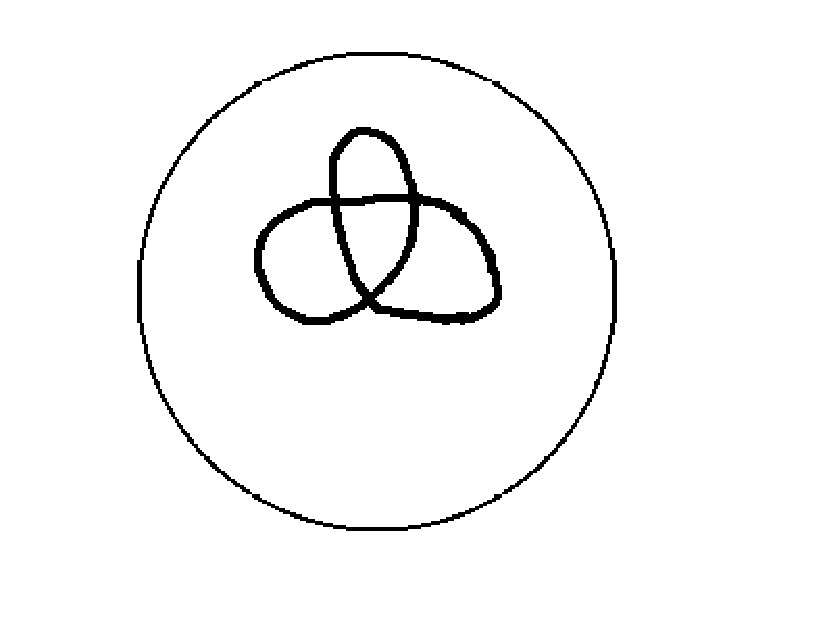}
\includegraphics[scale=0.5]{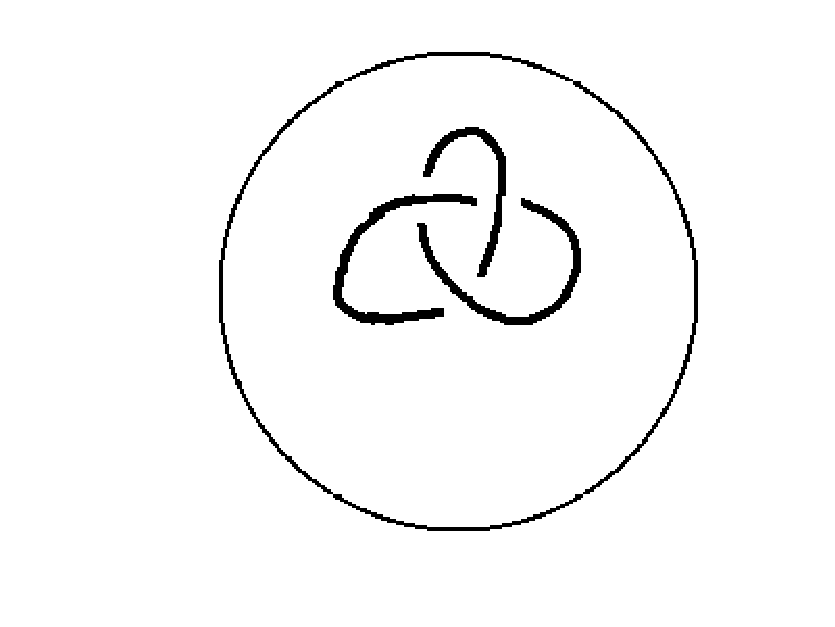}
\caption{}
\label{cross}
\end{figure}

The extension of these notions and of classical knot invariants to knots in$\R P^3$ was studied by Drobotukhina [1][2]:instead of projecting the knots to an affine plane they will be projected to a projective plane $\approx\R P^2\subset\R P^3$ in analogy with the $\R^3$-case. If $P:\R P^1\rightarrow \R P^3, P(s)=[x_0(s){:}x_1(s){:}x_2(s){:}x_3(s)]$ is a parametrization such that $[0{:}0{:}0{:}1]$ is not contained in the image, then projecting to a plane will correspond to ignoring the height function $x_3(s)$ and the image under the projection will be an immersion with any self intersections consisting of transversal double points. To draw diagrams we will use the representation of $\R P^2$ as $D^2$, see Section 2.1. Isotopies can again be described in terms of moves on diagrams. To deal with knots moving through the plane at infinity two additional moves are needed. See Fig. \ref{reidemeister} for a complete list of moves for diagrams of knots in $\R P^3$.

\begin{figure}
\includegraphics[scale=1.0]{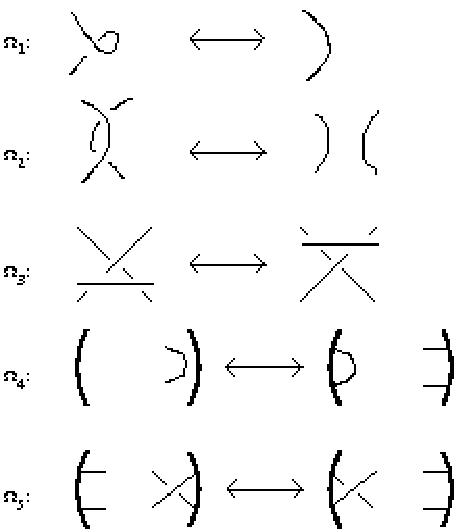}
\caption{}
\label{reidemeister}
\end{figure}

\begin{remark}
In knot theory in $\R^3$, the unknot plays a special role, being the only knot which has a corresponding diagram that can be transformed to a diagram without double points and so can be embedded in a plane. That role is played by two knots in $\R P^3$, the projective line and the standard circle which are the only two planar knots.
\end{remark}

\subsection{A gluing operation for rational knots}
The following theorem allows us to glue knots together while keeping track of the degree.

\begin{theorem} \label{combine}
If two knots admitting parametrizations of degree $m$ and $n$ intersect in only one point $p$ with linearly independent tangent lines at $p$, then it is possible to construct a knot parametrization $P$ with the following properties: $P$ is of degree $m+n$ and its image is isotopic to the union of the images of the two earlier parametrizations with a small perturbation at the intersection point $p$ as shown in Fig. \ref{perturb}.
\end{theorem}
\begin{proof}
Let $P=(p_0{:}p_1{:}p_2{:}p_3)$ and $P'=(p_0'{:}p_1'{:}p_2'{:}p_3')$ be parametrizations of knots $K$ and $K'$ of degrees $m$ and $n$ respectively, intersecting in a point $p$ as above. We may assume that $p=(1{:}0{:}0{:}0)$ and that $P(1{:}0)=P'(0{:}1)=p$. Consider the complexifications $\C K,\C K'$ and take a small closed ball $B$ around $p\in \C P^3$ and consider its complement $C$. We can without loss of generality assume that $P^{-1}(C)$ is an open set contained in an arbitrarily small $\epsilon-$ball around $(0{:}1)$ by acting on $\C P^1$ with a linear transformation sending $(s{:}t)$ to $(s{:}rt)$ for some suitably large $r\in\R_+$. Analogously $P'^{-1}(C)$ can be assumed to be contained in an arbitrarily small $\epsilon-$ball around $(1{:}0)$. Let $A$ be the union of the preimages of the plane at infinity under the two parametrizations $P,P'$. We construct rational functions $R,R'$ from $\C P^1\backslash A$ to $\C^3\subset \C P^3$ as follows: $$R(s{:}t)=(r_1(s{:}t),r_2(s{:}t),r_3(s{:}t))=\left(\frac{p_1(s{:}t)}{p_0(s{:}t)},\frac{p_2(s{:}t)}{p_0(s{:}t)},\frac{p_3(s{:}t)}{p_0(s{:}t)}\right)$$ $$R'(s{:}t)=(r_1'(s{:}t),r_2'(s{:}t),r_3'(s{:}t))=\left(\frac{p_1'(s{:}t)}{p_0'(s{:}t)},\frac{p_2'(s{:}t)}{p_0'(s{:}t)},\frac{p_3'(s{:}t)}{p_0'(s{:}t)}\right).$$ We construct a new rational function $$\hat{R}(s{:}t)=R(s{:}t)+R'(s{:}t).$$ The function $\hat{R}$ will then be a rational function of degree $m+n$.\\
For points close to $(0{:}1)$ we have $\hat{R}=R+v$ where $v$ is a function taking values in $B$. It follows that the image will agree with the image of a section of a tubular neighborhood of $K\backslash B$. Similarly, for points close to $(1{:}0)$ the image will agree with the image of a section of a tubular neighborhood of $K'\backslash B$. All other points will map into a small ball around $(1{:}0{:}0{:}0)$ (being the sum of two vectors in $B\subset\C^3\subset \C P^3$) and it follows by inspection that the map gives us a resolution of the intersection as shown in Fig. \ref{perturb}. We extend $\hat{R}$ to a rational knot parametrization $\hat{P}: \C P^1\to \C P^3$ of degree $m+n$ by the natural inclusion $\C^3\subset\C P^3$ and by continuity at points contained in $A$. The resulting knot parametrization will then satisfy the properties of the theorem. \end{proof}

\begin{figure}
\includegraphics[scale=0.5]{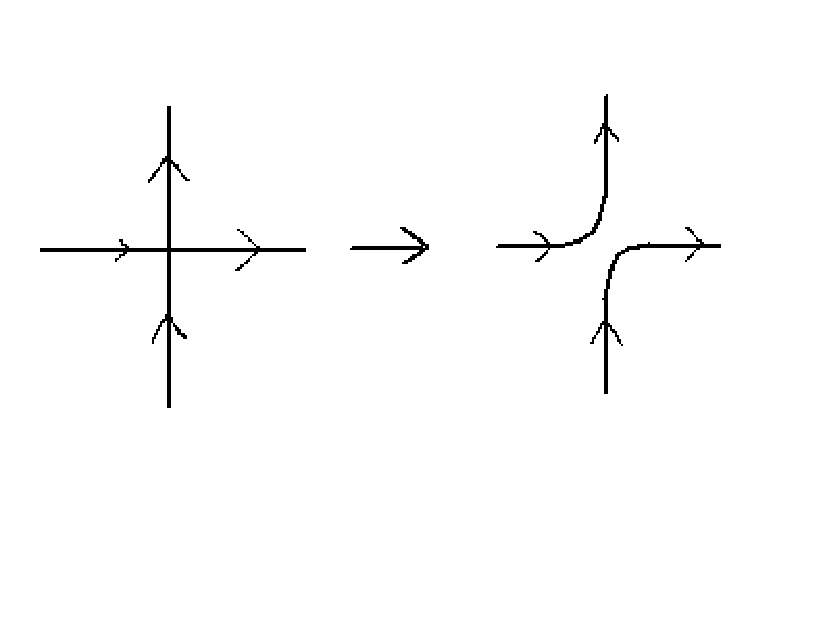}
\caption{}
\label{perturb}
\end{figure}

\section{Classification of knots of low degree}

In this section the proofs for the main results of this paper are presented.

\subsection{Rigid isotopy classification for degrees 1-4}

Recall the real algebraic knot invariant writhe $w$ mentioned in Section 1. We will use the following results from \cite{VIRO}.

\begin{lemma}
For an algebraic knot $K$ of degree $d$, $$\frac{(d-1)(d-2)}{2}=w(K) \mod{2}.$$
\end{lemma}
\begin{proof}
After projection to a generic plane an algebraic knot $K$ of degree $d$ has $\frac{(d-1)(d-2)}{2}$ double points when counting both real, solitary and nonreal double points. Each real or solitary double point either increases or decreases writhe by one and each nonreal double point $p$ comes with a conjugate point $\bar{p}$ (neither which is assigned a number since they are nonreal) so $\frac{(d-1)(d-2)}{2}=w(K) \mod{2}$.
\end{proof}
\begin{lemma}
For an algebraic knot $K$ of degree $d$, $$-\frac{(d-1)(d-2)}{2}\leq w(K)\leq\frac{(d-1)(d-2)}{2}.$$
\end{lemma}
\begin{proof}
After projection to a generic plane an algebraic knot $K$ of degree $d$ has $\frac{(d-1)(d-2)}{2}$ double points when counting both real and complex double points. Each double point contributes at most $1$ and at least $-1$ to the writhe.
\end{proof}
\begin{remark}
The limits in these inequalities are attained, see subsection \ref{kdw}.
\end{remark}

We describe two simple tools for examining $\mathcal{K}_d$. We can act on $\mathcal{K}_d$, the space of rational knot parametrizations of degree $d$ with $PSL_4(\R)$ in a natural way. Since this Lie group is connected, such action moves any point only within its connected component. This allows us to perform a form of Gauss elimination on the coefficients of parametrizing polynomials of a real algebraic knot by rigid isotopy. We especially note that if we rotate in an appropriate plane we can send $(a{:}b{:}c{:}d)$ to $(b{:}-a{:}c{:}d)$ or $(-a{:}-b{:}c{:}d)$. This allows us to move signs between the coordinates, and to permute the polynomials of parametrizations (up to some signs).\\
Our method for examining rigid isotopy for knots of higher degrees than $3$ is to consider the space $\mathcal{C}_d$ of rational curve parametrizations of degree $d$. This space is connected and contains the space $\mathcal{K}_d$  of rational knot parametrizations of degree $d$. The complement of $\mathcal{K}_d$ in $\mathcal{C}_d$ is the discriminant $\Sigma$ which is a stratified space $$\Sigma=\Sigma^1\supset\Sigma^2\supset\cdots\supset\Sigma^k.$$
To study the components of $\mathcal{K}_d$ we examine the components of the highest strata $\Sigma^1\setminus\Sigma^2$of codimension $1$. We call the components of this strata {\em walls}. The walls consists of curves with a single self intersection with linearly independent tangent vectors at the intersection point. A generic path connecting two knots in $\mathcal{C}_d$ will only intersect the discriminant in walls and at each such intersection it will be transversal to the wall.

\subsection{Lemmas used in the proof of Theorem \ref{rigidclass}}
In this subsection lemmas used for the proof of those parts of Theorem \ref{rigidclass} which pertain to knots of degree $d\leq4$ are stated.

\begin{lemma}
 There are the following 4 rigid isotopy classes of rational knots of degree $d\leq3$.
\begin{enumerate}
\item For $d=1$ there is only one rigid isotopy class:
\begin{itemize}
 \item the line with $w=0$.
\end{itemize}
\item For $d=2$ there is only one rigid isotopy class:
\begin{itemize}
 \item the circle with $w=0$.
\end{itemize}

\item For $d=3$ there are two rigid isotopy classes:
\begin{itemize}
 \item a line with $w=1$ and
 \item a line with $w=-1$.
\end{itemize}
\end{enumerate}
\end{lemma}
\begin{proof}
We examine the degrees separately.
\begin{itemize}
\item Knots of degree 1 are lines and so obviously rigidly isotopic.

\item Knots of degree two are planar and must be a second degree curve in this plane. Only one such curve exists after a (linear) reparametrization, namely a standard circle.

\item Consider the space of curve parametrizations and a knot parametrization $P$. Each parametrization is given by a $4\times4$ matrix, with entry $a_{ij}$ corresponding to the coefficient of the $t^{j-1}s^{4-j}$ term of the parametrizing polynomial $p_{i-1}$. If this matrix is not invertible, the knot must be planar. But planar curves of degree $3$ has one self intersection, and so cannot be knots. If the matrix is invertible then $P$ will truly be a knot parametrization so $\mathcal{K}_3$  is homeomorphic to $PGL_4(\R)$. This space has two connected components, we can choose representatives $P(t{:}s)=(t^3{:}st^2{:}s^2t{:}\pm s^3)$ from them having encomplexed writhe $w=\pm1$. However, the knots are smoothly isotopic to each other, and to a straight line.

\end{itemize}
\end{proof}

Degree $4$ requires examining the walls separating components of $\mathcal{K}_4$, the following lemma gives a bound on the number of walls.

\begin{lemma}
The highest strata of the discriminant for degree $d=4$, $\Sigma_4^1\slash\Sigma_4^2$ consists of at most three components.
\end{lemma}
\begin{proof}
We examine the elements of $\mathcal{C}_4$ that parametrizes curves with only one self intersection. We note that if a curve would have more than one self intersection, it would be planar since a plane containing these two points of self intersection and a third point from the knot would intersect it $5$ times (counted with multiplicity).
A self intersection can occur in two ways, either from two real points or from two (conjugate) nonreal points. After a linear transformation, we can assume that the point of self intersection is $p=(0{:}0{:}0{:}1)$. We now have two cases to examine.
\begin{enumerate}
 \item One real self intersection
 \item One solitary self intersection arising from two nonreal points
\end{enumerate}

We begin with the first case, (1).
Let the doublepoint occur at the image of $(0{:}1)$ and $(1{:}0)$ and let the image be $p=(0{:}0{:}0{:}1)$. The parametrization looks as follows:
$P(t{:}s)=(ts\cdot p_1(t,s){:}ts\cdot p_2(t,s){:}ts\cdot p_3(t,s){:}p_4(t,s))$ where $p_i$ is a polynomial of degree 2 for $i=1..3$. Since the knot is nonplanar the first three polynomials must be linearly independent. We can use Gauss elimination to put the parametrization on a standard form $P(t{:}s)=(ts^3{:}t^2s^2{:}t^3s{:}p_4(t,s))$ (any signs are moved to $p_4$). We can again apply an element from $PSL_4(\R)$ to remove all terms except $t^4, s^4$ from $p_4$ (both terms must exist with nonzero coefficients since not all four polynomials can have a common factor). Normalizing leaves us with the following four possibilities:
$$P_1(t{:}s)=(ts^3{:}t^2s^2{:}t^3s{:}t^4+s^4)$$
$$P_2(t{:}s)=(ts^3{:}t^2s^2{:}t^3s{:}t^4-s^4)$$
$$P_3(t{:}s)=(ts^3{:}t^2s^2,t^3s{:}-t^4-s^4)$$
$$P_4(t{:}s)=(ts^3{:}t^2s^2,t^3s{:}-t^4+s^4)$$
The second and fourth options cross infinity and correspond to walls between twocrossing knots and the unknot, one separating a twocrossing knot with writhe $3$ from the unknot with writhe $1$ and one from the mirror image of a twocrossing knot with writhe $-3$ to the unknot with writhe $-1$.
The first and third options are mirror images, however they are in the same wall isotopic as can be seen by constructing a path between them. By taking $P_1$ and sending $(t,s)$ to $(-s,t)$ by acting with an appropriate element from $PSL_2(\R)$ we get $(-t^3s{:}t^2s^2{:}-ts^3{:}t^4+s^4)$. Permuting polynomials and moving signs gives  $(ts^3{:}t^2s^2{:}-t^3s{:}t^4+s^4)=P_3(t{:}s)$. This wall separates knots with encomplexed writhe $w=1$ from knots with encomplexed writhe $w=-1$ and correspond to the wall between the two options for the unknot (encomplexed writhe $w=\pm 1$). So in total we have at most three components.\\

It is left to examine the second case (2), a solitary self intersection arising from two nonreal points.

It is left to examine the case where two complex conjugate branches intersect, we can assume that $P(i{:}1)=P(-i{:}1)=(0{:}0{:}0{:}1)$. We again look at the parametrization and note that the first three polynomials have a factor $t^2+s^2$. A similar argument as before gives us:
$P(t{:}s)=(t^2s^2+s^4{:}t^3s+ts^3{:}t^4+t^2s^2{:}ats^3+bs^4)$
Some slight reparametrization gives
$P(t{:}s)=((t^2+s^2)^2{:}(t^2+s^2)(t^2-s^2){:}2ts(t^2+s^2){:}ats^3+bs^4)$.
This knot is planar if and only if $a=b=0$. We can continuously change $a,b$ until $a=0,b=1$. So there is only one wall arising from a solitary self intersection.

After projecting to a plane this is clearly a circle without double points (the doublepoint we had before was arising from nonreal parts and was situated at infinity so it is not visible). So the only doublepoint is the one arising from ($\pm i$) and so the knots at both sides of this wall must have writhe $\pm 1$ and be the trivial knots.
Both of these walls separate the same pair of knots as can be seen by degenerating them to each other, at the edge connecting this solitary doublepoint wall and the real doublepoint wall we will have a planar curve with a cusp, exactly the same situation as when a real doublepoint turns to a solitary doublepoint.

\end{proof}

\begin{corollary}
 There are the following 4 rigid isotopy classes of rational knots of degree $d=4$.
\begin{itemize}
 \item The twocrossing knot with writhe $3$.
 \item An unknot with writhe $1$.
 \item An unknot with writhe $-1$.
 \item The mirror image of the twocrossing knot with writhe $-3$.
\end{itemize}
\end{corollary}

\begin{remark}
This gives us a complete rigid and smooth isotopy classification of knots of degree $\leq4$, by just ignoring the demand of rigidity we easily see that up to smooth isotopy the possible knots are the line, the circle, the twocrossing knot and its mirror image since any knot must be adjacent to one of our three walls.
\end{remark}


\begin{remark}
A reason for why knots of degree $\leq4$ are relatively easy to find (compared to knots for higher degrees) is that they lie on quadrics (which are fairly simple to understand) while knots of higher degrees lies surfaces of degree 3 or higher in general.
\end{remark}

Previously we examined the highest strata of the discriminant (consisting of curves having a self intersection) of the space of knots of degree 4. For degree 5 we do something similar, but instead of examining walls, we examine the components of the second highest strata ($\Sigma_d^2\setminus\Sigma_d^3$) called {\em edges} where walls meet, which consists of curves with two self intersections. The following lemma will be useful in examining degree 5.

\begin{lemma}
A curve of degree 5 having 3 distinct self intersections is planar.
\end{lemma}
\begin{proof}
Take a plane through the three points where it selfintersects. Counting with multiplicity the curve will intersect the plane at least 6 times, and so must be contained in it.
\end{proof}

We need to study the edges (of codimension 2) where walls intersect. Each knot is adjacent to such an edge as the following lemma shows.

\begin{lemma}\label{closeedge}
On the boundary of each component in $\mathcal{K}_5\subset\mathcal{C}_5$ there is a curve with just two self intersections.
\end{lemma}

\begin{proof}
Take an arbitrary knot parametrization of degree $5$ in $\mathcal{K}_5$. Any other knot will either be rigidly isotopic to it (so there exists a path connecting them not intersecting the discriminant) or they will not be rigidly isotopic, meaning there will be walls between them. Take a point on such a wall. This will be a curve with one self intersection. Project the corresponding curve to a plane in such a way that it has at least two double points after projection. Consider the rational function describing the height over this plane. It will be a rational function of degree 5. We keep the denominator fixed so as not to disturb the image on the plane. We add the restriction that the height function will take the same value at our self intersection and some other pair of points that are not projected to the same point. Another dimension disappears from scaling. This leaves us three dimensions. By moving in these three dimensions we will not end up in a plane since our second pair of points have the same relative nonzero height and so we cannot end up anywhere planar. We move so that we get another self intersection (possible since we just make sure to move in such a direction that the relative heights between our pair of points being projected to a second double point decreases). We will then end up at an edge, so each knot parametrization is adjacent to an edge consisting of a curve with two self intersections.
\end{proof}

The following three lemmas will describe the different possibilities for the edges. We will consider a path in the strata $\Sigma^1\setminus\Sigma^2$ to be a {\em wall isotopy} while a path in the strata $\Sigma^2\setminus\Sigma^3$ will be called an {\em edge isotopy}.

\begin{lemma}\label{2solitary}
A curve parametrization of degree 5 with two solitary self intersections arising from nonreal points in the preimage is edge isotopic to
$$P(t{:}s)=((t^2+1)(t^2+4)t{:}(t^2+s^2)(t^2+4s^2)s{:}(t^2+s^2)t^2s{:}(t^2+4s^2)ts^2).$$
\end{lemma}

\begin{proof}
By acting on $\R P^3$ with a linear transformation we place the intersections at $(0{:}0{:}0{:}1)$ and $(0{:}0{:}1{:}0)$. We will also assume that the points which end up at the selfintersection are $(1{:}z),(1{:}\bar{z}),(1{:}w)$ and $(1{:}\bar{w})$. The parametrization will look as follows:
$$x_0(t,s)=(t-zs)(t-\bar{z}s)(t-ws)(t-\bar{w}s)(at+bs)$$
$$x_1(t,s)=(t-zs)(t-\bar{z}s)(t-ws)(t-\bar{w}s)(ct+ds)$$
$$x_2(t,s)=(t-zs)(t-\bar{z}s)p(t,s)$$
$$x_3(t,s)=(t-ws)(t-\bar{w}s)q(t,s)$$
Since it was assumed that the knot is nonplanar the first two coordinates cannot be linearly dependent, so it follows that we can put them in the form where $a=d=1,c=b=0$ by a linear transformation (any eventual sign is moved to $x_3$. The polynomial $x_0$ has a leading $t^5$ term with coefficient one while $x_1$ has a term $s^5$ with nonzero coefficient (namely, $|zw|^2$) which we use to eliminate all $t^5$ and $s^5$ terms from the last two polynomials. This give us a slightly nicer form for the parametrization:
$$x_0(t,s)=(t-zs)(t-\bar{z}s)(t-ws)(t-\bar{w}s)t$$
$$x_1(t,s)=(t-zs)(t-\bar{z}s)(t-ws)(t-\bar{w}s)s$$
$$x_2(t,s)=(t-zs)(t-\bar{z}s)(at^2s+bts^2)=(t-zs)(t-\bar{z}s)ts(at+bs)$$
$$x_3(t,s)=(t-ws)(t-\bar{w}s)(ct^2s+dts^2)=(t-ws)(t-\bar{w}s)ts(ct+ds).$$
As long as $|zw|^2\neq0$ and $x_2$,$x_2$ are linearly independent the knot is nonplanar and so stays on our edge. If $x_2$ and $x_3$ were to be linearly dependent then either $a=b=0$ or $c=d=0$. As long as that is avoided, the knot is nonplanar. This is not a strong demand since $\R^2\backslash\{(0,0)\}$ is connected. Similarly we can change $z,w$ as long as we take care not to let them become real or equal. This allows us to transform this edge into a canonical form:
$$x_0(t,s)=(t^2+1)(t^2+4)t$$
$$x_1(t,s)=(t^2+1)(t^2+4)s$$
$$x_2(t,s)=(t^2+1)t^2s$$
$$x_3(t,s)=(t^2+4)ts^2.$$
\end{proof}

\begin{remark}

It is especially interesting to note that this canonical form was independent on the orientation of the knot (that is, if we had switched knot orientation in the beginning we would still be edge isotopic to this canonical form) and multiplying the last coordinate with $-1$. Since taking the mirror image means that writhe switches sign, and we cannot distinguish a signshift, this edge must be separating a pair of walls between knots with writhe $0$, $0$, $2$ and $-2$ since this is the only possibly configuration of writhes around the edge that is not disturbed by a shifting of signs. We can also move $z$ toward the real line, letting it and $\bar{z}$ split into two real points causing a self intersection as long as we stay away from the situation when $z=0$. this means that at least one of the cases of one real, one nonreal pair self intersecting can be transformed into this canonical form.
\end{remark}

\begin{lemma}\label{1sol1nonsol}
A curve parametrization of degree 5 with one solitary self intersections and one non solitary self intersection arising from two nonreal points and two real points in the preimage is edge isotopic to
$$P(t{:}s)=((t^2+1)(t^2+4)t{:}(t^2+s^2)(t^2+4s^2)s{:}(t^2+s^2)t^2s{:}(t^2+4s^2)ts^2).$$
\end{lemma}
\begin{proof}

After a linear transformation of $\R P^3$ and $\R P^1$ we can assume that the real pair is $(0{:}1)$ and $(1{:}0)$ and the complex pair is $(1{:}z)$ and $(1{:}\bar{z})$. They are sent to $(0{:}0{:}0{:}1)$ and $(0{:}0{:}1{:}0)$ respectively. This gives us a parametrization on the following form:
$$x_0(s,t)=ts(t-zs)(t-\bar{z}s)(at+bs)$$
$$x_1(s,t)=ts(t-zs)(t-\bar{z}s)(ct+ds)$$
$$x_2(s,t)=tsp(s,t)$$
$$x_3(s,t)=(t-zs)(t-\bar{z}s)q(s,t).$$
Since the knot was nonplanar, $x_0$ and $x_1$ must be linearly independent so after a linear transformation (and possibly moving a sign to $x_3$) we can assume that $a=d=1,b=c=0$. Then $x_0$ has a term $st^4$ with coefficient $1$ and $x_0$ has a term $s^4t$ with coefficient $|z|^2$. They can be used to eliminate those terms from $x_2$ and $x_3$ giving us something on the following form:
$$x_0(s,t)=t^2s(t-zs)(t-\bar{z}s)$$
$$x_1(s,t)=ts^2(t-zs)(t-\bar{z}s)$$
$$x_2(s,t)=ts(at^2s+bts^2)=s^2t^2(at+bs)$$
$$x_3(s,t)=(t-zs)(t-\bar{z}s)(ct^3+ds^3).$$
The four polynomials are linearly independent as long as $x_2$,$x_3$ are linearly independent since $x_0$ is the only polynomial with $st^4$ term, and $x_1$ is the only polynomial with $s^4t$. The polynomials $x_2$ and $x_3$  are obviously linearly independent as long as $(a,b)\neq(0,0)$ and $(c,d)\neq(0,0)$. This is however not enough since all the polynomials cannot be allowed to have a common zero since they parametrize something in the projective space. The only possible common zeroes of all four polynomials are $(1,z),(1,\bar{z}),(1,0)$ and $(0,1)$. This means that neither $c$ nor $d$ can be zero. On the other hand, as long as $a,b$ are not both zero at once, we can change them at will so we assume that $a=1$, $b=0$. We can normalize $c,d$ such that $c=\pm 1$, $d=\pm1$. This gives us four cases. However, at least one of those cases arise from the double solitary pair case, when letting one self intersection become real. For this case we know that is is independent of sign-switch of a polynomial and orientation-switch of the knot. Assume $(c,d)=(x,y)$ is this case. Then the case $(c,d)=(-x,-y)$ is also situated on the same edge to it, since it is obtained by switching the sign of the last coordinate. If we instead switch orientation of the knot (sending $(s,t)$ to $(s,-t)$) we obtain the following form:
$$x_0(s,t)=t^2s(t+zs)(t+\bar{z}s)$$
$$x_1(s,t)=ts^2(t+zs)(t+\bar{z}s)$$
$$x_2(s,t)=-s^2t^3$$
$$x_3(s,t)=(-t-zs)(-t-\bar{z}s)(-xt^3+ys^3).$$
By moving $(z,\bar{z})$ to $(-\bar(z),-z)$ and switching the signs of the last two coordinate with a rotation we obtain
$$x_0(s,t)=t^2s(t-zs)(t-\bar{z}s)$$
$$x_1(s,t)=ts^2(t-zs)(t-\bar{z}s)$$
$$x_2(s,t)=s^2t^3$$
$$x_3(s,t)=(t-zs)(t-\bar{z}s)(-xt^3+ys^3).$$
That is, the case $(c,d)=(-x,y)$. Applying the same trick to the case $(c,d)=(-x,-y)$ we at last get the case $(c.d)=(x,-y)$ giving us the result that there is only one edge with one real and one solitary pair, which is the same as edge of curves with two solitary self intersections and so can be assigned the same canonical form. Again we note that we can move our $z$ until it hits the real line (staying away from $0$) to get a case of two real pairs situated on the same edge as this one and so get the same canonical form as lemma \ref{2solitary}.
\end{proof}

\begin{lemma}\label{2nonsol}
A curve parametrization of degree 5 with two non solitary self intersections is edge isotopic to either
$$P(t{:}s)=((t^2+1)(t^2+4)t{:}(t^2+s^2)(t^2+4s^2)s{:}(t^2+s^2)t^2s{:}(t^2+4s^2)ts^2),$$
$$P(t{:}s)=(t^2s(t-s)(3t-s){:}ts^2(t-s)(3t-s){:}\pm t^2s^2(3t-2s){:}(t-s)(3t-s)(t^3+s^3))$$ or
$$P(t{:}s)=(t^2s(t^2-s^2){:}\pm ts^2(t^2-s^2){:}t^2s^3,(t^2-s^2)(t^3\pm s^3)).$$
Where the first two possibilities are independent of orientation.
\end{lemma}

\begin{proof}
We again assume (after a linear transformation of both $\R P^1$ and $\R P^3$) that we have points $(1{:}0)$, $(0{:}1)$ that are sent to $(0{:}0{:}0{:}1)$ while $(1{:}1)$ and $(1{:}c)$ (with $c\neq1,0$) is sent to $(0{:}0{:}1{:}0)$. This gives us the following form:
$$x_0(s,t)=ts(t-s)(t-cs)(at+bs)$$
$$x_1(s,t)=ts(t-s)(t-cs)(ct+ds)$$
$$x_2(s,t)=tsp(s,t)$$
$$x_3(s,t)=(t-s)(t-cs)q(s,t).$$
The same argument as earlier allows us to put $a=d=1,b=c=0$. This gives us one term $t^4s$ with coefficient $1$ in $x_0$ and one term $ts^4$ with coefficient $c\neq0$. We use these terms to remove all such terms from $x_2,x_3$ giving us something on the following form:
$$x_0(s,t)=t^2s(t-s)(t-cs)$$
$$x_1(s,t)=ts^2(t-s)(t-cs)$$
$$x_2(s,t)=ts(at^2s+bts^2)=t^2s^2(at+bs)$$
$$x_3(s,t)=(t-s)(t-cs)(dt^3+es^3).$$
Since we do not allow common roots of all four polynomials $d\neq0$. We can rescale the fourth polynomial and move any signs to the third to obtain $d=1$. Then we have
$$x_0(s,t)=t^2s(t-s)(t-cs)$$
$$x_1(s,t)=ts^2(t-s)(t-cs)$$
$$x_2(s,t)=t^2s^2(at+bs)$$
$$x_3(s,t)=(t-s)(t-cs)(t^3+es^3).$$
As earlier, the two first polynomials are linearly independent and we need only to watch out for common roots and  the two last polynomials being linearly dependent. They are however clearly linearly independent as long as $(a,b)\neq(0,0)$

This gives us several different cases to examine. Maybe the most interesting difference is whether $c>0$ or not, since $c>0$ gives us the possibility that our edge is the same as the edge where two solitary self intersections appeared.\\
We begin with the case $c>0$. By rescaling and relabeling the parametrization if necessary we may also assume that $c<1$.
The only further limitation on $c$ is that $(1,c)$ is not allowed to be a root of $(at+bs)$. If we can move $c$ to $1$ without hitting a root of $(at+bs)$ we can let the two self intersections become solitary and be back in an earlier case.
Otherwise we know that any eventual root $(1,f)$ of $(at+bs)$ is such that $c<f<1$. We now examine eventual roots of $t^3+es^3$. If $e<0$ we can move the roots $(0,1)$ and $(1,0)$ towards each other until they unite and become complex, again returning us to an earlier case. So we can from hereon assume that $e>0$ and so let $e=1$. Also, since $(0,1)$ was not a root of $(at+bs)$ we can put $a=\pm 1$. We now have
$$x_0(s,t)=t^2s(t-s)(t-cs)$$
$$x_1(s,t)=ts^2(t-s)(t-cs)$$
$$x_2(s,t)=\pm t^2s^2(t-fs)$$
$$x_3(s,t)=(t-s)(t-cs)(t^3+s^3)$$
with $0<c<f<1$.
We can move $f,c$ until we have $c=1/3,f=2/3$ giving us two cases that are mirror images of each other:
$$x_0(s,t)=t^2s(t-s)(3t-s)$$
$$x_1(s,t)=ts^2(t-s)(3t-s)$$
$$x_2(s,t)=\pm t^2s^2(3t-2s)$$
$$x_3(s,t)=(t-s)(3t-s)(t^3+s^3).$$
This edge will separate knots of writhe $\pm2,\pm4,\pm4,\pm6$, taking the mirror image will of course switch the sign. It is independent of orientation.
We also have the other case, $c<0$, again with the form:
$$x_0(s,t)=t^2s(t-s)(t-cs)$$
$$x_1(s,t)=ts^2(t-s)(t-cs)$$
$$x_2(s,t)=t^2s^2(at+bs)$$
$$x_3(s,t)=(t-s)(t-cs)(t^3+es^3).$$
Again the only demands on $a,b,c$ are that $(a,b)\neq(0,0)$, $c>0$ and $ac+b\neq0$. We have two cases to examine. If $ac+b<0$ we can move $b$ towards $-\infty$ and then move $c$ to $-1$ while keeping $ac+b<0$. If $ac+b>0$ we can do the same but with $b$ becoming some big positive number. So we can assume that $c=-1$. The demand will then become that $a\neq b$. We also have the demand that $e\neq0$. Our new form for the polynomials are:
$$x_0(s,t)=t^2s(t^2-s^2)$$
$$x_1(s,t)=ts^2(t^2-s^2)$$
$$x_2(s,t)=t^2s^2(at+bs)$$
$$x_3(s,t)=(t^2-s^2)(t^3+es^3).$$
Since $a\neq b$ we can move $a$ to zero and $b$ to either $1$ or $-1$, since our only demand on $e$ is that $e\neq0$ we can normalize it to let $e=\pm 1$.
$$x_0(s,t)=t^2s(t^2-s^2)$$
$$x_1(s,t)=ts^2(t^2-s^2)$$
$$x_2(s,t)=bt^2s^3$$
$$x_3(s,t)=(t^2-s^2)(t^3+es^3).$$
This gives us four options, $b=\pm 1,e=\pm 1$. Starting with the case $b=e=1$ we can get the other three cases by switching orientation of the knot and/or taking the mirror image. Taking the mirrorimage will switch the writhe, so this will truly give us a new edge. This edge will separate knots of writhes $\pm0,\pm2,\pm2,\pm4$, the sign will switch when taking the mirror image.\\
\end{proof}

\begin{remark}
This classification immediately gives us a complete smooth isotopy classification of knots of degree five by just perturbing our canonical forms (resolving the self intersections in one of two ways each) giving explicit parametrizations from each possible component of the space of knots of degree 5.
\end{remark}

Rigid isotopy classification needs some further work. By examining the canonical forms, it is clear that we have the following (possibly) different edges, the edge arising from two nonreal pairs in the preimage is separating knots with writhe $0,0,2,-2$, we have at most two edges separating writhe $0,2,2,4$ and one separating $2,4,4,6$. We also have their mirror images, however a rigid isotopy classification of knots with non-negative writhe would for symmetry reasons give a complete rigid isotopy classification. We will consider curves on a wall separating knots with writhe $k-1$ and $k+1$ to have writhe $k$.

\begin{lemma}\label{slide}
Given a wall of writhe $w>0$ (that is, separating knots of writhe $w-1,w+1$) in the space of curves of degree $5$ it is adjacent to an edge separating knots with writhe $w+1,w-1,w-1,w-3$.
\end{lemma}
\begin{proof}
Choose a wall with writhe $w>0$ and choose a projection to a plane. This projection will have at most $2$ double points with negative writhe and it will also have a doublepoint arising from the self intersection. By again varying the height function keeping these three points at the same relative height we can (without ending up in a plane) disturb our wall until we end up at an edge. One of the projected double points will then have to be from a self intersection, since we kept that from happening at the points with negative writhe one of the double points with positive writhe must have turned into a self intersection.
\end{proof}
\begin{remark}
The lemma says that we can always {\em slide} towards an edge of lower writhe. This enables us to connect our edges with the help of the adjacent walls.
\end{remark}

\begin{proof}[Proof of the second part of Theorem \ref{rigidclass}]
In the proof we will use Lemma \ref{slide} several times to go from a wall separating knots with positive writhes to an edge adjacent to it separating knots with lower writhes. We will refer to this process as {\em sliding}. By Lemma \ref{closeedge} we know that any knot component lies adjacent to an edge with two self intersections, so it is enough to classify the knots close to edges up to rigid isotopy.
By Lemma \ref{2solitary}, Lemma \ref{1sol1nonsol} and Lemma \ref{2nonsol} we have four cases to examine if we search for a rigid isotopy classification for knots with nonnegative writhe. A simple examination reveals (at most) two edges separating $0,2,2,4$ which we denote with $E,E'$. The difference between them was choice of orientation. We also had one edge separating knots with writhe $-2,0,0,2$ which we call $F$ and finally one edge separating knots of writhe $2,4,4,6$ which we call $G$. Both $F,G$ were independent of orientation. We begin by proving that there is only one knot (up to rigid isotopy) with writhe $2$. Take a wall close to the $E$ edge separating knots with writhe $0,2$. Slide it towards the $-2,0,0,2$ edge. The $2$ writhe knots from $E$ and from the $-2,0,0,2$ edge must then be the same. We use the same argument on $E'$:s two writhe $2$ knots to show that they are all the same. We take a wall separating knots with writhe $2,4$ close to the $2,4,4,6$ edge and slide it to wards either $E$ or $E'$. Wherever we end up, the writhe $2$ knot must join up with the writhe $2$ knot from $E,E'$. So there is but one knot of writhe $2$. It is trivial that there is but one knot of writhe $6$. We now prove that there is only one knot (again up to rigid isotopy) of writhe $4$. We examine the two edges $E,E'$ closer. We recall that the last polynomial $x_1(s,t)=(t^2-s^2)(t^3\pm s^3)$ was what separated them. Perturb the crossing arising from $(0,1)$ and $(1,0)$ slightly so that we end up at a wall separating $2,4$.

The parametrization then looks as follows:
$$x_0(s,t)=t^2s(t^2-s^2)$$
$$x_1(s,t)=(\epsilon s^2+ts)s(t^2-s^2)$$
$$x_2(s,t)=t^2s^3$$
$$x_3(s,t)=(t^2-s^2)(t^3+us^3)$$
with $u=\pm1$ depending on if we are close to $E$ or $E'$. Now we just move along the wall by letting $u$ go from $1$ to $-1$. We end up at the other edge implying that the two $4$ writhe knots arising from $E,E'$ are really one and the same. It is left to show that they are the same knot as the two $4$ writhe knots arising from the $2,4,4,6$ edge. Choose a on a wall separating one of these writhe $4$ knots from a writhe $2$ knot and slide towards lower writhe, then we end up at either $E$ or $E'$, whatever the case, we connect our writhe $4$ knots. The only thing left to show is that there is only one knot of writhe $0$. We repeat the argument from writhe $4$ to show that the two writhe $0$ parts arising from $E,E'$ are the same (by just choosing an opposite sign of $\epsilon$. So if we slide towards the $-2,0,0,2$ edge from a wall close to those knots we will end up with an orientation independent writhe $0$ knot close to the $-2,0,0,2$ edge. We study this edge closer.

$$x_0(t,s)=(t^2+s^2)(t^2+4s^2)t$$
$$x_1(t,s)=(t^2+s^2)(t^2+4s^2)s$$
$$x_2(t,s)=(t^2+s^2)t^2s$$
$$x_3(t,s)=(t^2+4s^2)ts^2.$$
The two knots with writhe zero could then be represented as
$$x_0(t,s)=(t^2+\mp\epsilon ts+s^2)(t^2 \pm\epsilon ts+4s^2)t$$
$$x_1(t,s)=(t^2+s^2)(t^2+4s^2)s$$
$$x_2(t,s)=(t^2+s^2)t^2s$$
$$x_3(t,s)=(t^2+4s^2)ts^2.$$
Note that they differ by a change of orientation and we knew that one of them (the one we got from sliding from $E,E'$) was orientation independent so they must be in the same component.

Analogously the arguments can be applied to the mirror images to get the same result for negative writhes. Since there is one and only one knot (considered up to rigid isotopy) of each possible writhe we have proved the theorem.
\end{proof}

\subsection{The knots $K_d^w$}\label{kdw}
In this section we give an algorithm to construct certain knots. We denote these knots $K_d^w$. It turns out that they realize each possible writhe $w$ for a given degree $d$.

We construct a collection of oriented lines as follows:\\
Look at the affine part $\R^3\subset\R P^3$. Place a straight line vertically oriented upwards which has very slight slope. Call this line $L_1$. Inductively we construct $d-1$ more lines. Take an identical copy of line $L_n$ , turn it slightly clockwise and give it a slightly higher slope and move it along the $L_n$ line slightly upwards. Let this be $L_{n+1}$. Do this for $d$ lines.\\
This construction gives $d$ lines, where each line intersects another line. By resolving these intersections we will get a rational knot of degree $d$. It will have writhe $\frac{(d-1)(d-2)}{2}$ since there is that many double points after a projection, and at each such point we will have positive writhe. Let this knot be $K_d^{\frac{(d-1)(d-2)}{2}}$. An example for $d=4$ is given in Fig. \ref{K43}.\\
\begin{figure}\label{K43}
\includegraphics[scale=0.35]{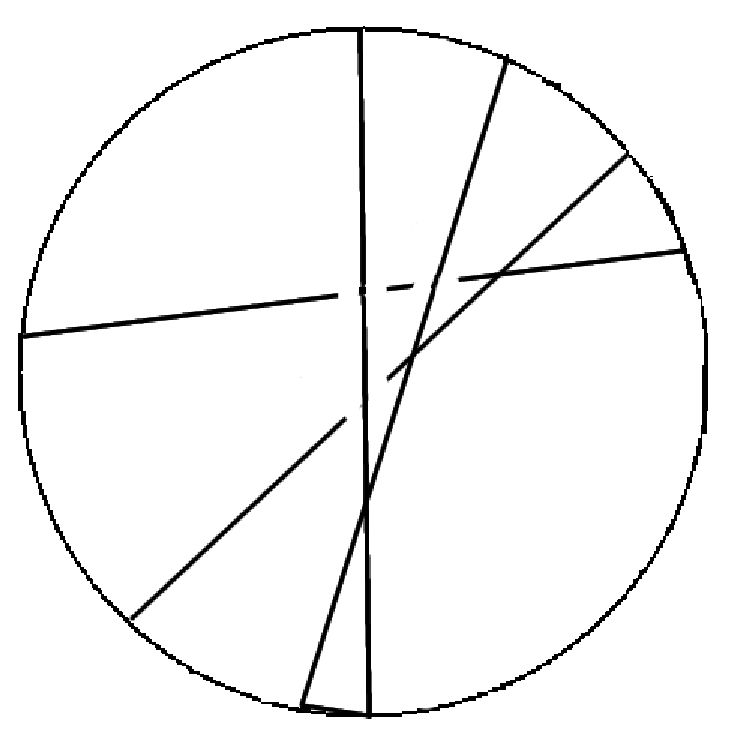}
\includegraphics[scale=0.35]{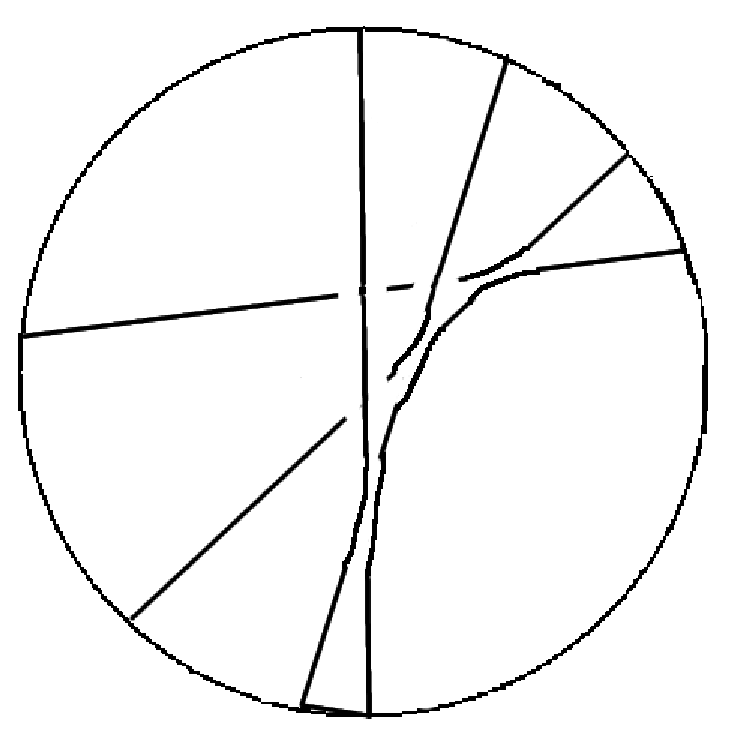}
\includegraphics[scale=0.35]{2crossingknot.eps}
\caption{The construction of $K_3^4$}
\end{figure}

We use the same construction of lines but change the slope of the last line while keeping the intersection fixed with the second last line. Each time it crosses one of the other lines the writhe (after an eventual resolution) decreases by two and we get a new knot. We keep doing this until it has passed all the other lines and has a highly negative slope. Take the second to last line and repeat until you have crossed all lines except the last. Each time we cross one of the other lines the writhe will decrease by two. Continue this process with the second last line keeping intersections with the last and third last line by just moving the last line vertically and then continue. In this way we will go $2$ steps at a time from the least writhe to the highest writhe (when all crossings have a positive writhe number) and so realize each possible writhe number $w$ and then name the resulting knot as $K_d^w$. Since we have used $d$ lines (which are of degree 1) the knot gluing operation then states that this is a rational knot of degree $d$, and by the construction it will have writhe $w$. It is obvious that any two members of such a class are rigidly isotopic since the only thing left to our choice is the individual slopes and points of intersection which can gradually be changed.

\begin{proof}[Proof of Theorem \ref{allwrithe}]
Follows trivially from the construction of $K_w^d$.
\end{proof}
\begin{remark}
The construction is not very symmetric so one cannot be sure that the mirror image of $K_d^w$ is necessarily (rigidly isotopic to) $K_d^{-w}$.
\end{remark}

\begin{remark}
Since $K_d^w$ attains each possible writhe for the given degree $d$ this proves that the lemmas stated earlier are sharp.
\end{remark}
\begin{remark}
Since we have found all rigid isotopy classes of degree $d\leq5$ we can conclude that $K_d^w$ are the only rigid isotopy classes of degree $d\leq5$.
\end{remark}
\begin{remark}\label{deg6counter}
However the knots $K_d^w$ are not the only knots possible (for higher degrees) so writhe and degree together does not tell us everything about a knot. Figure \ref{counter6} gives an example of two 6th degree knots, both with the same writhe but not isotopic to each other and hence not rigidly isotopic.
\end{remark}

\begin{figure}\label{counter6}
\includegraphics[scale=0.35]{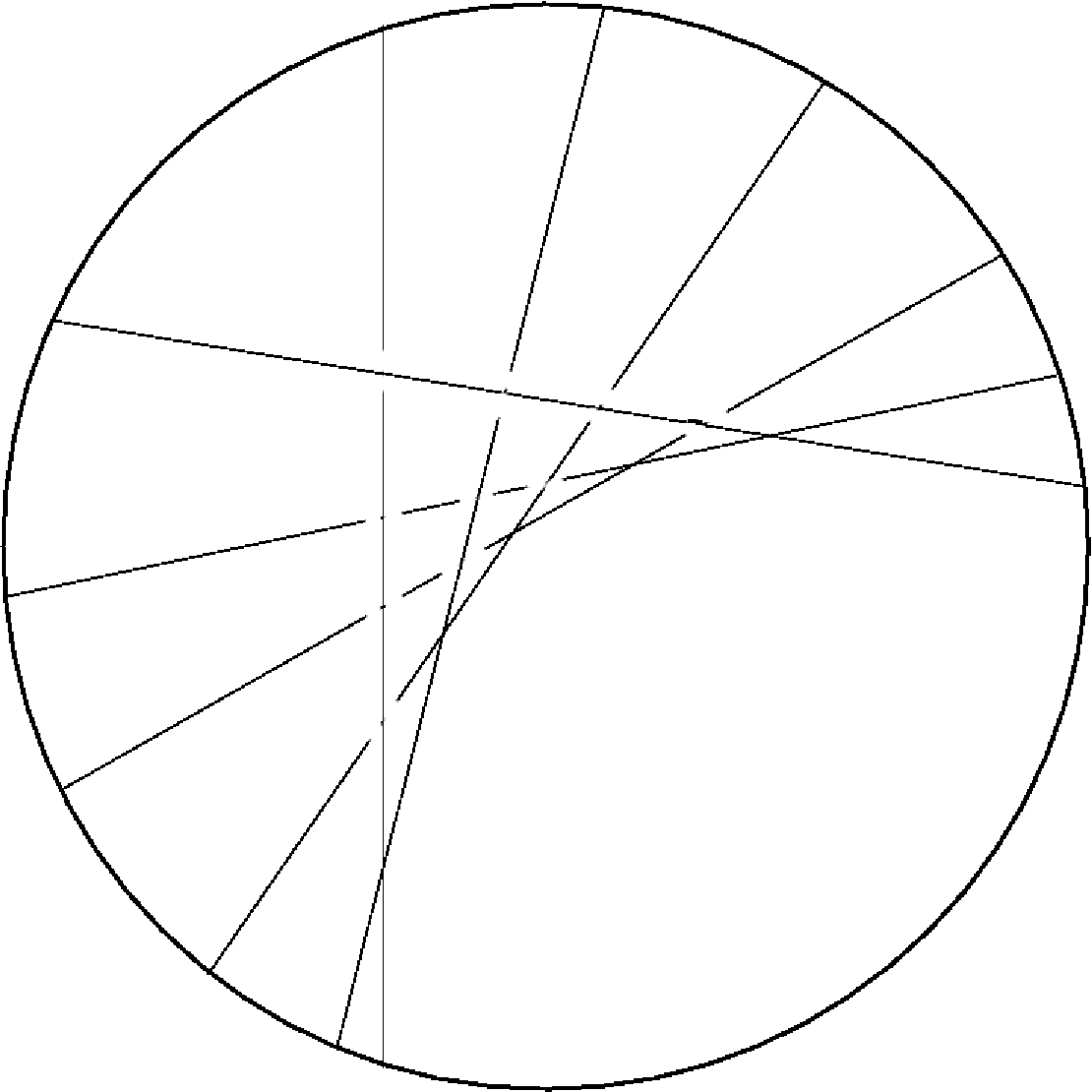}
\includegraphics[scale=0.35]{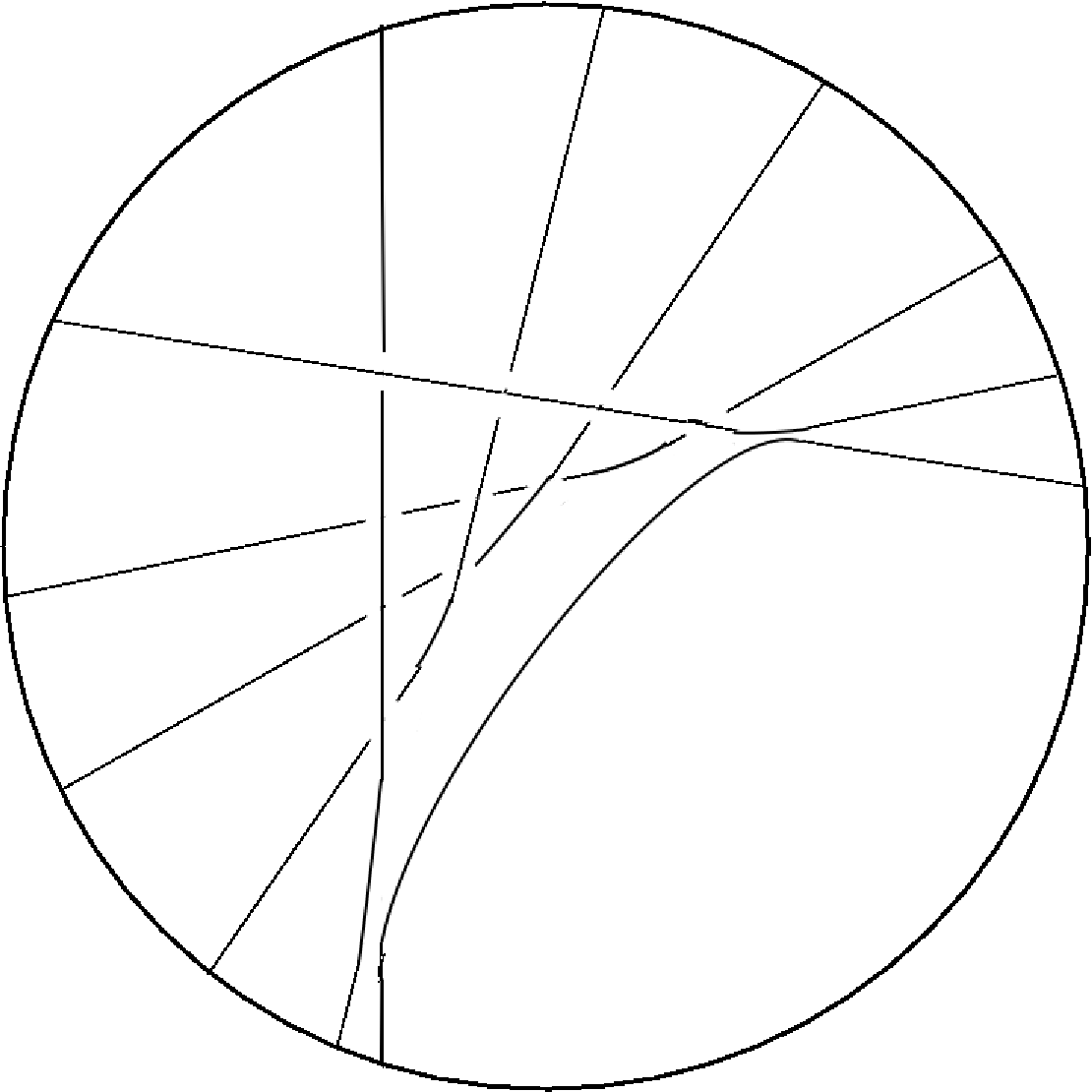}
\includegraphics[scale=0.35]{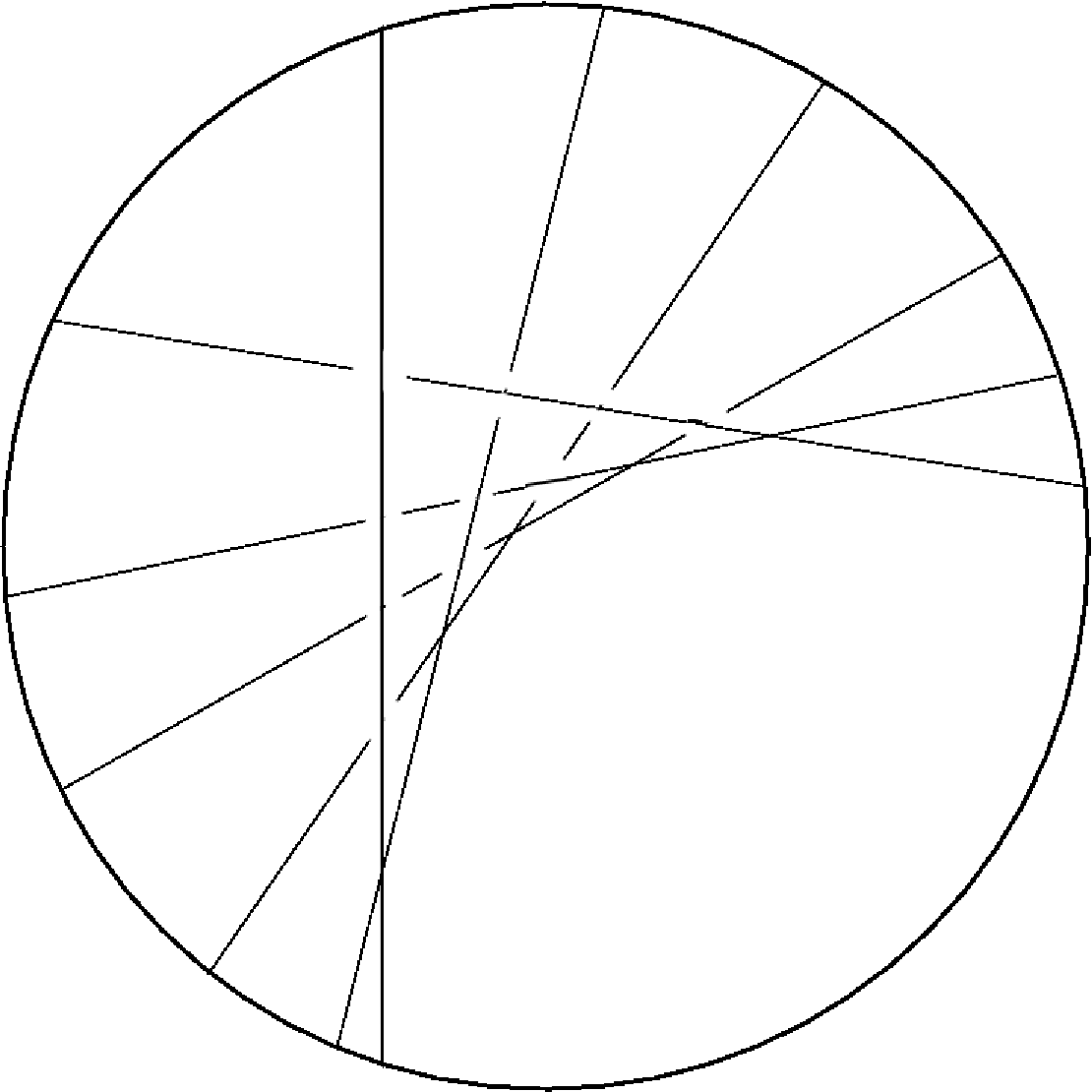}
\includegraphics[scale=0.35]{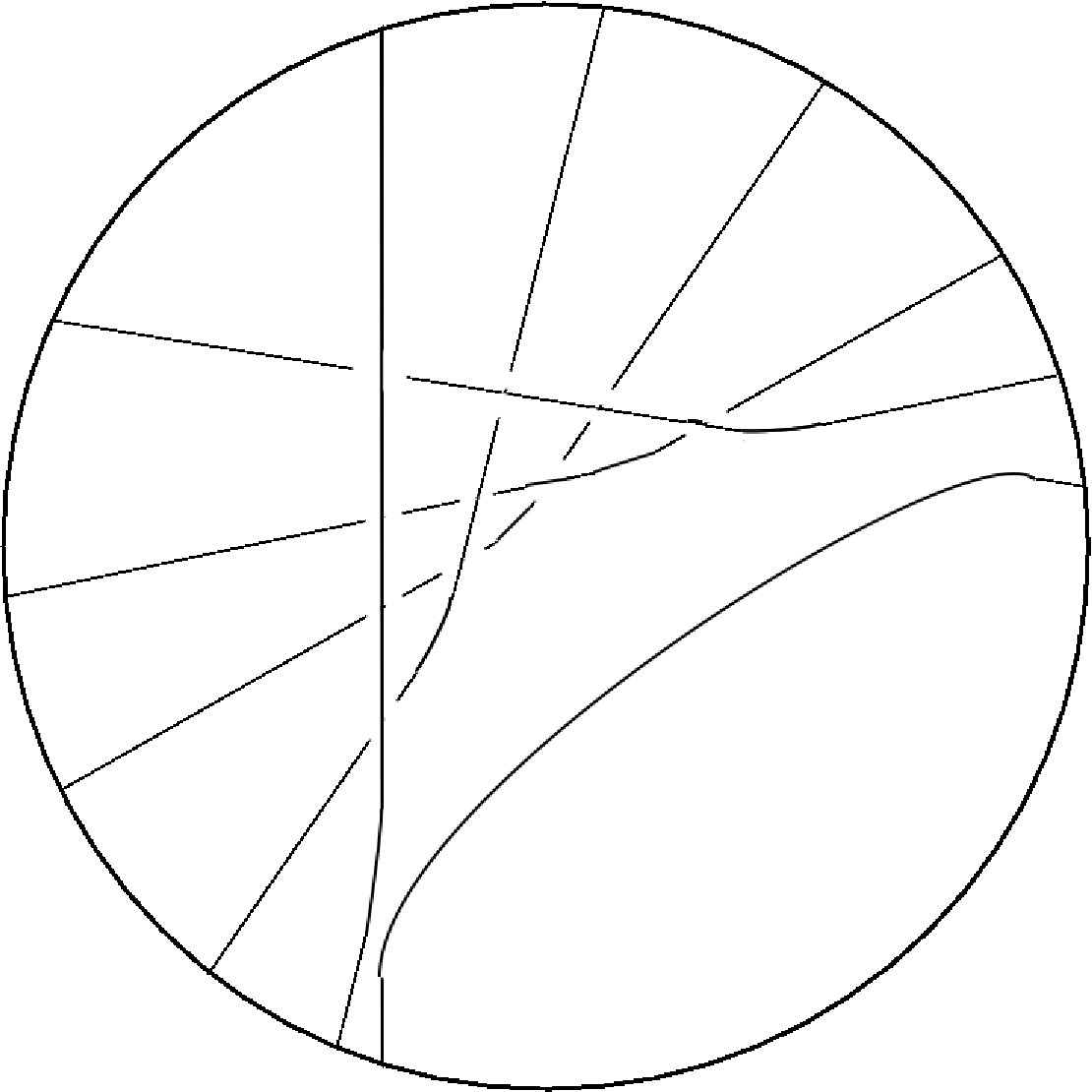}
\caption{Writhe is not enough to distinguish knots in degree 6}
\end{figure}

\clearpage
\subsection{Topological isotopy classification}

Since rigid isotopy implies topological isotopy we can just use Theorem \ref{rigidclass} to get a collection of all possible knots of degree $d\leq5$ and then consider them up to topological isotopy for a smooth isotopy classification.


\begin{theorem}\label{4knots}
All irreducible knots with diagrams that has three or four crossings that cannot be transformed to diagrams with fewer crossings have rational degree 5 or 6.
\end{theorem}
\begin{proof}
Since they are not topologically isotopic to a knot with 2 crossings or less they cannot have rational degree $\leq4$ so they must have degree larger than 4 with parity of degree decided by the parity of the homology. We construct all such irreducible knots with degree 5 or 6 in the next section.
\end{proof}

\section{Diagrams and constructions}
\subsection{Method of construction}
Two methods will be used to construct the knots. One of them is explicit parametrization. It is rather easy to get explicit parametrizations from our earlier edges by just perturbing them slightly. It is also possible to construct all the possible knots out of lines and the gluing operation from Theorem \ref{combine}. Together these two methods give all the constructions needed for Theorem \ref{main}. If necessary, explicit constructions could be recovered.
Knot diagrams and the gluing operation will be used without much explaining text. If a knot of degree $a$ and a knot of degree $b$ are combined we will write Degree $a+b$ without calculating the sum. After applying Theorem \ref{combine} the same notation with a computed sum will be used for the degree of the new knot.

\clearpage
\subsection{Knots with up to four crossings}
Here irreducible knots with 3 or 4 crossings are shown to be parametrizable by degree 5 or 6. We simply construct all affine and projective irreducible knots in the tables of Drobotukhina and Rolfsen with 3 and 4 crossings. We already know that the 2-crossing knot, the two planar knots, the long trefoil, the projective $5_3-$knot and their mirror images are parametrizable by degree $\leq5$. So it is left to check that the remaining knots with up to 4 crossings are realizable with degree 6.
\begin{figure}
\includegraphics[scale=0.35]{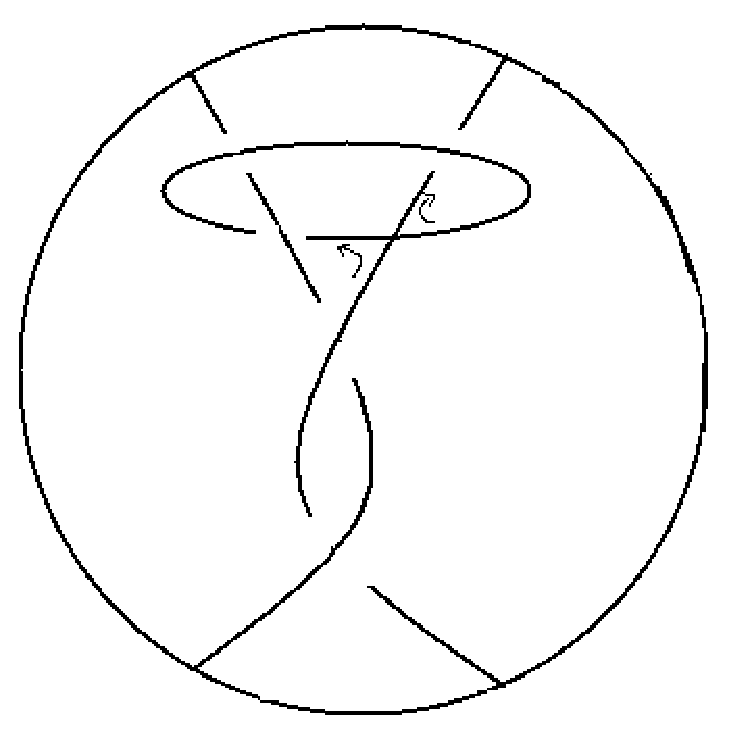}
\includegraphics[scale=0.35]{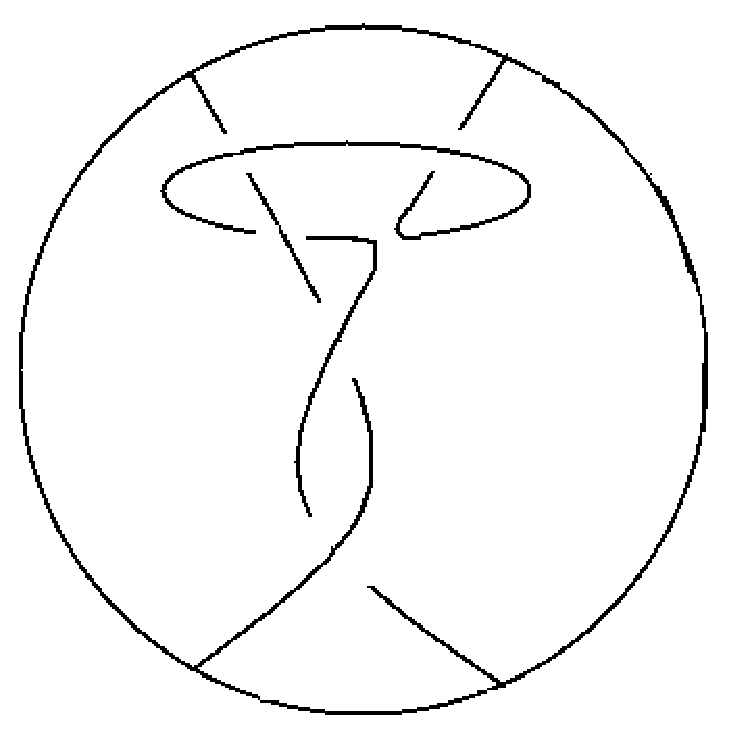}
\includegraphics[scale=0.35]{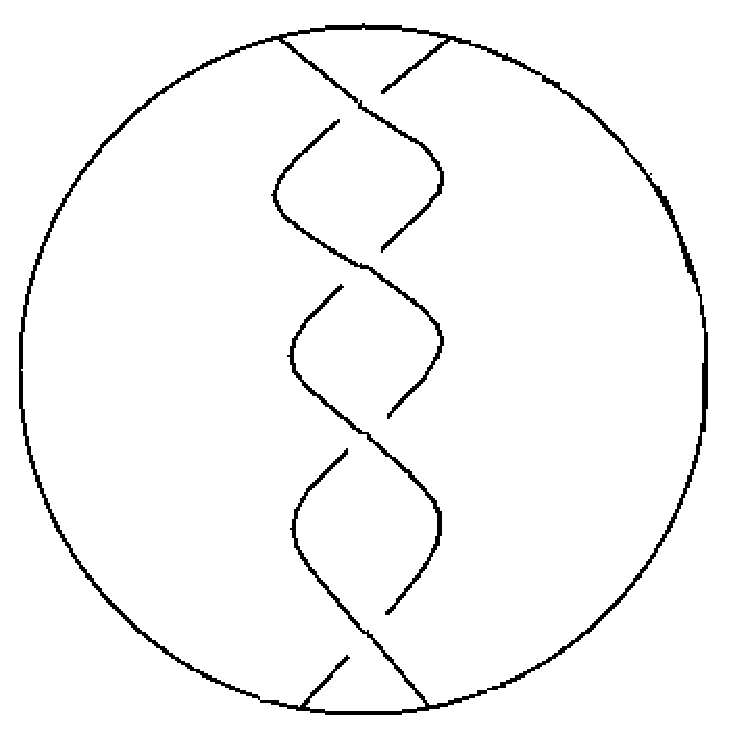}
\caption{Degree 4+2 = Degree 6}
\end{figure}

\begin{figure}
\includegraphics[scale=0.35]{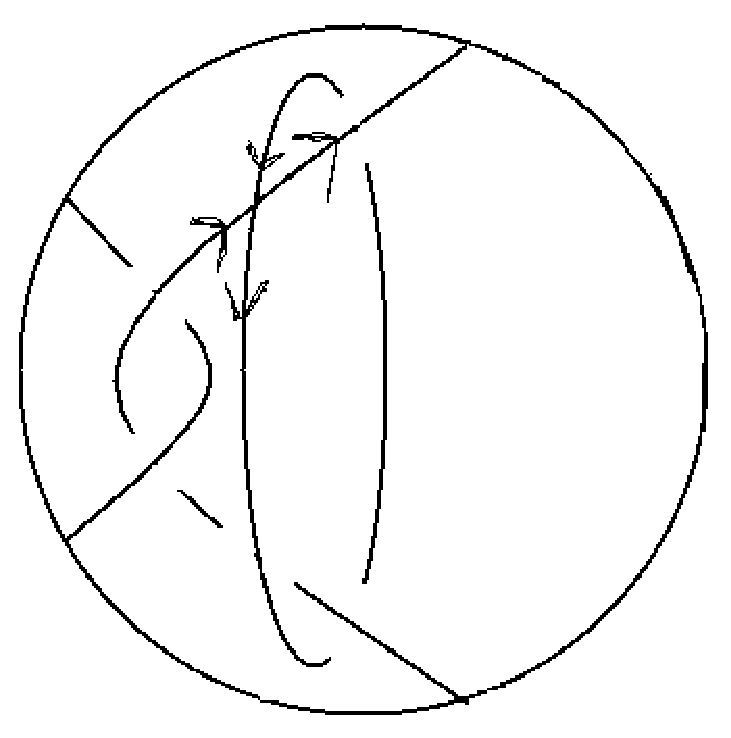}
\includegraphics[scale=0.35]{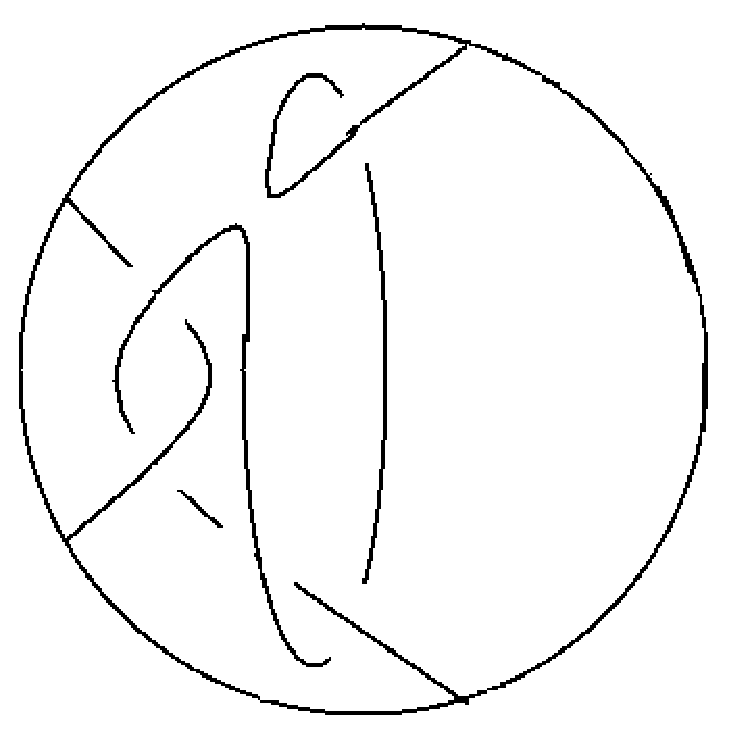}
\includegraphics[scale=0.35]{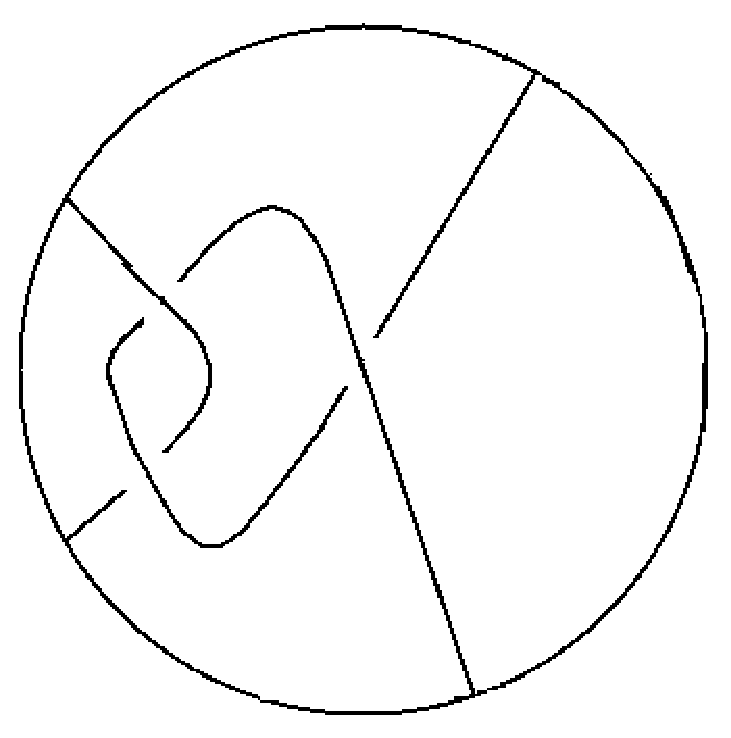}
\caption{Degree 4+2 = Degree 6}
\end{figure}

\begin{figure}
\includegraphics[scale=0.35]{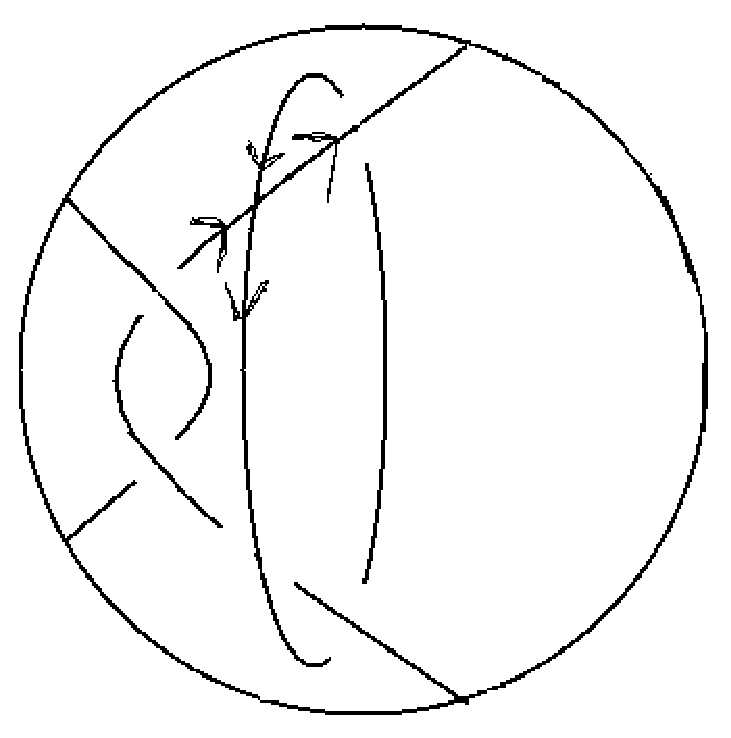}
\includegraphics[scale=0.35]{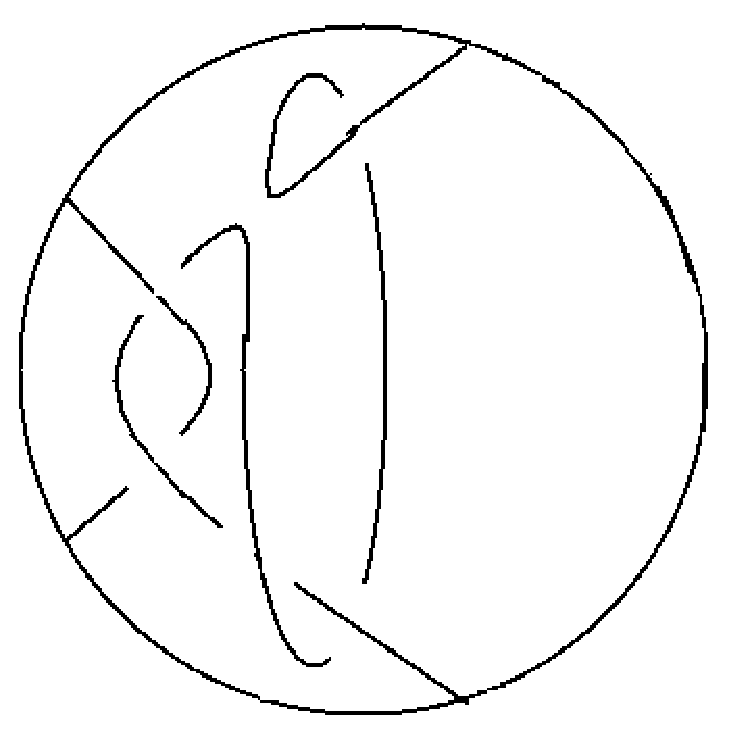}
\includegraphics[scale=0.35]{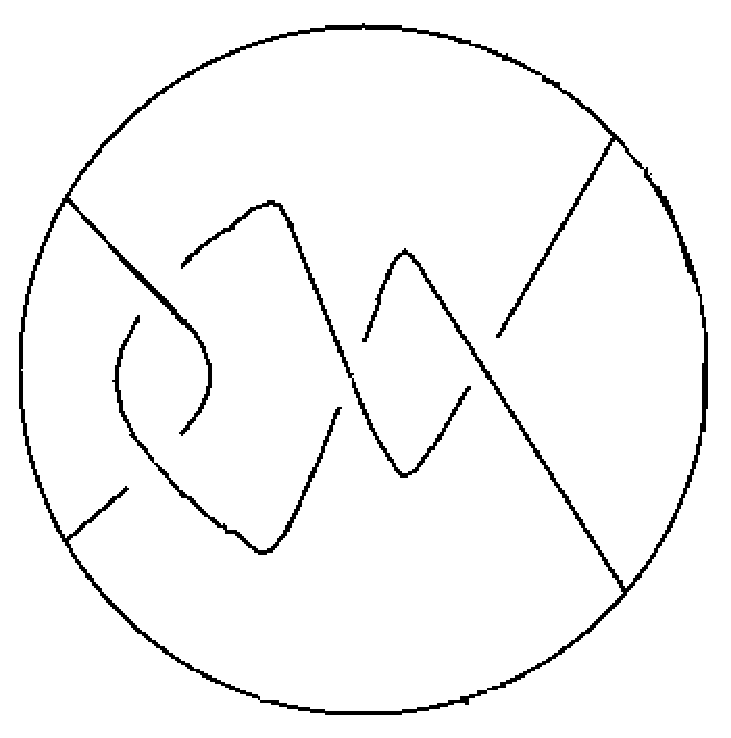}
\caption{Degree 4+2 = Degree 6}
\end{figure}

\begin{figure}
\includegraphics[scale=0.35]{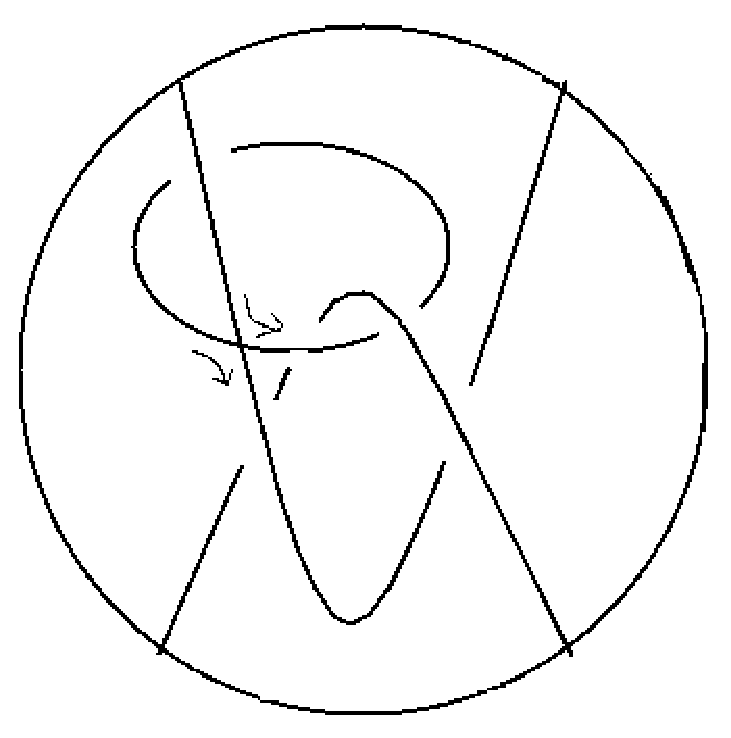}
\includegraphics[scale=0.35]{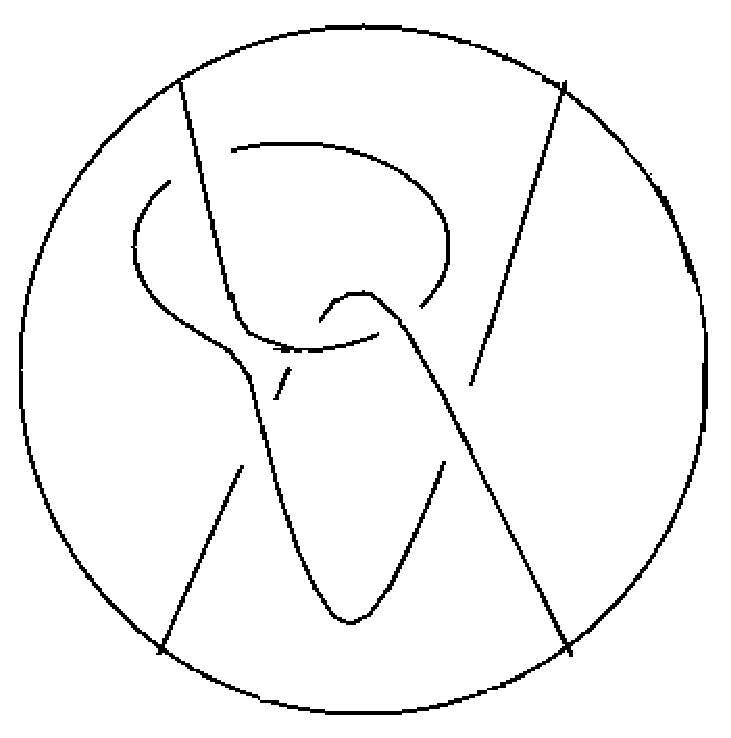}
\includegraphics[scale=0.35]{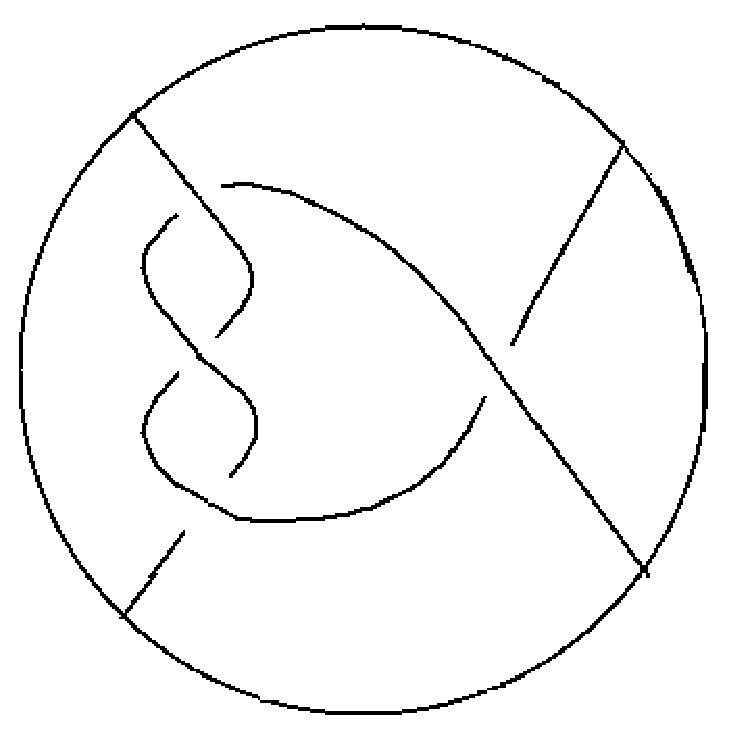}
\caption{Degree 4+2 = Degree 6}
\end{figure}

\begin{figure}
\includegraphics[scale=0.35]{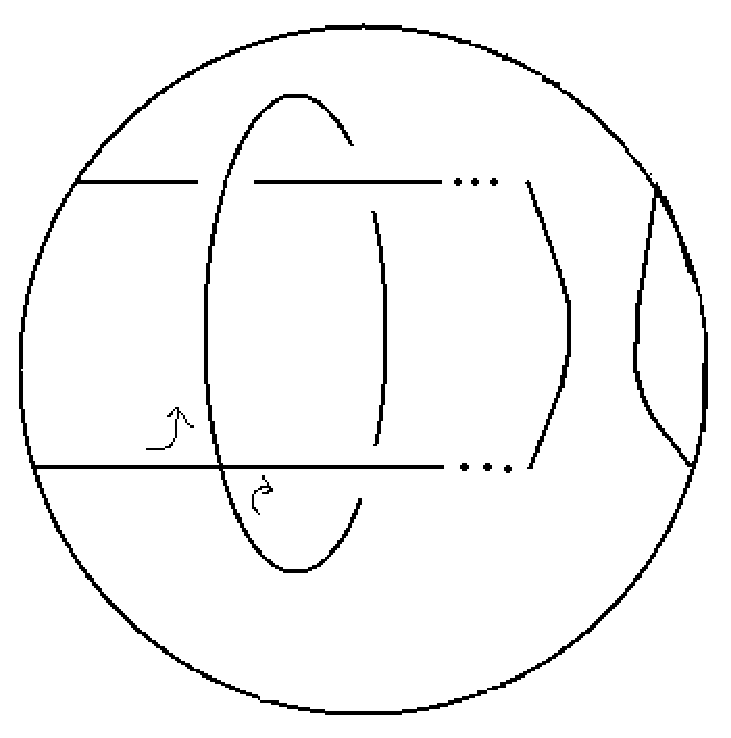}
\includegraphics[scale=0.35]{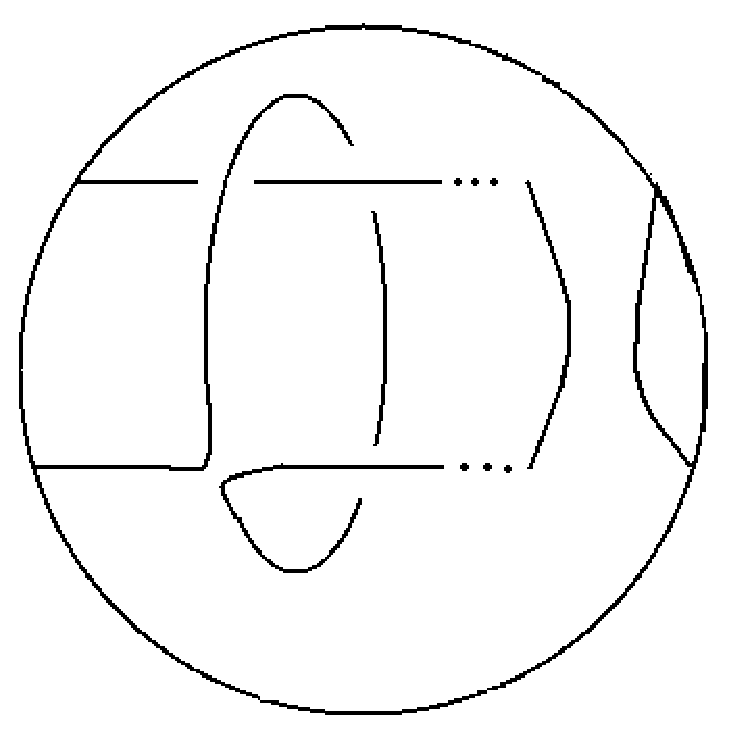}
\includegraphics[scale=0.35]{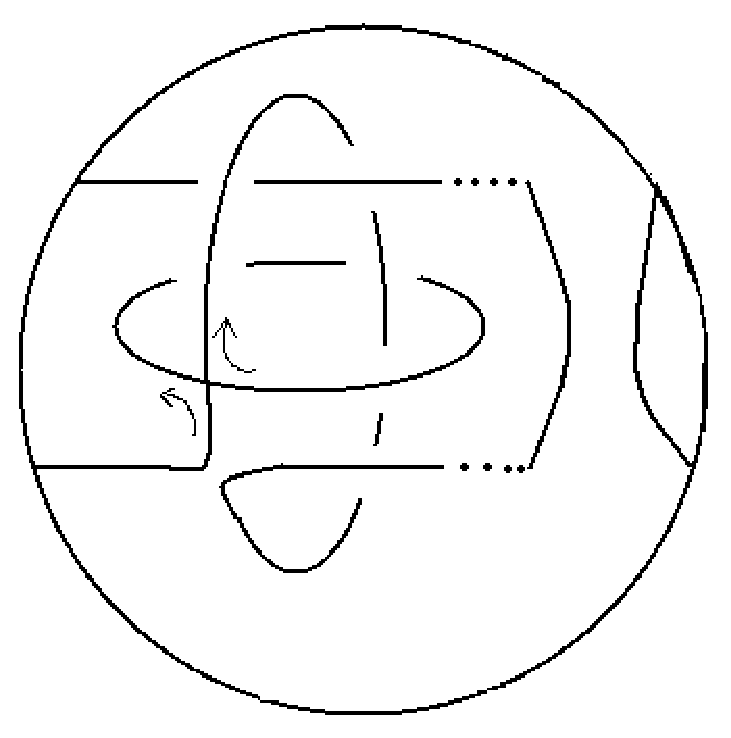}
\includegraphics[scale=0.35]{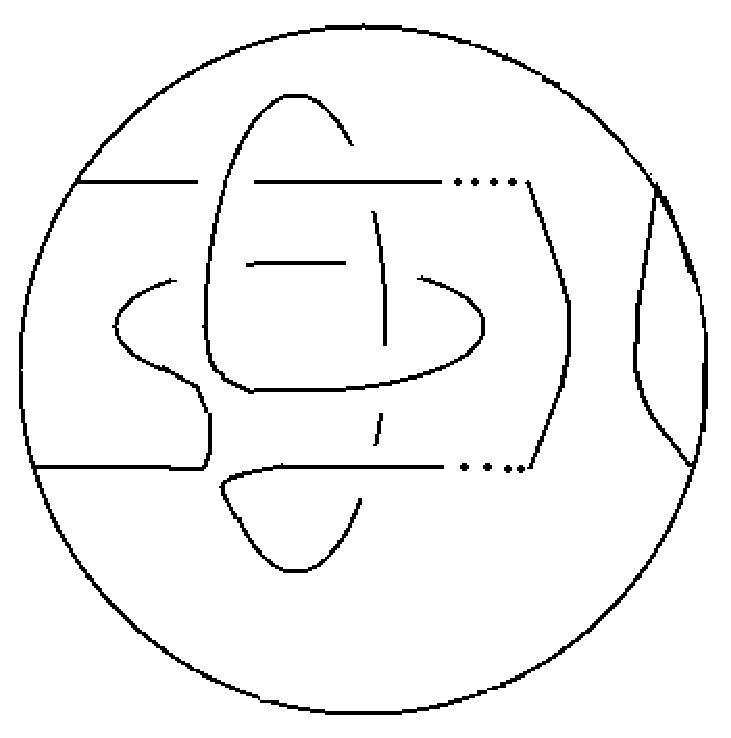}
\includegraphics[scale=0.35]{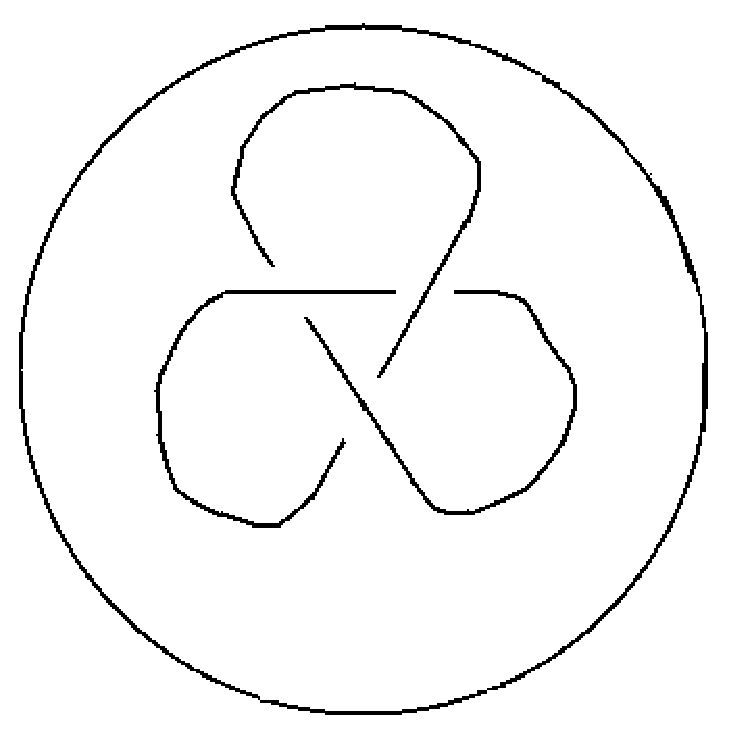}
\caption{Degree 2+2 = Degree 4, Degree 4+2 = Degree 6}
\end{figure}
\begin{figure}
\includegraphics[scale=0.35]{affinestartcomb.eps}
\includegraphics[scale=0.35]{affinestartcombd.eps}
\includegraphics[scale=0.35]{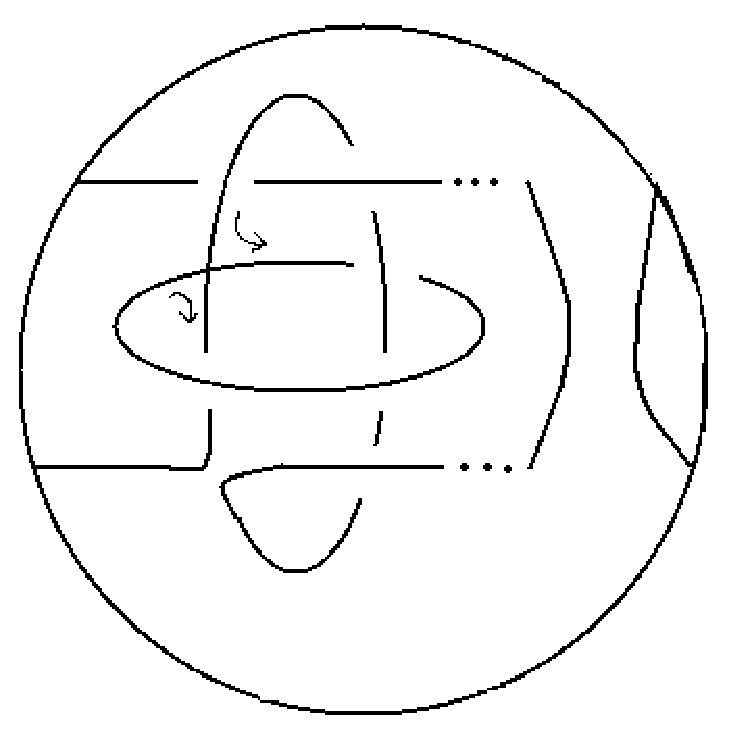}
\includegraphics[scale=0.35]{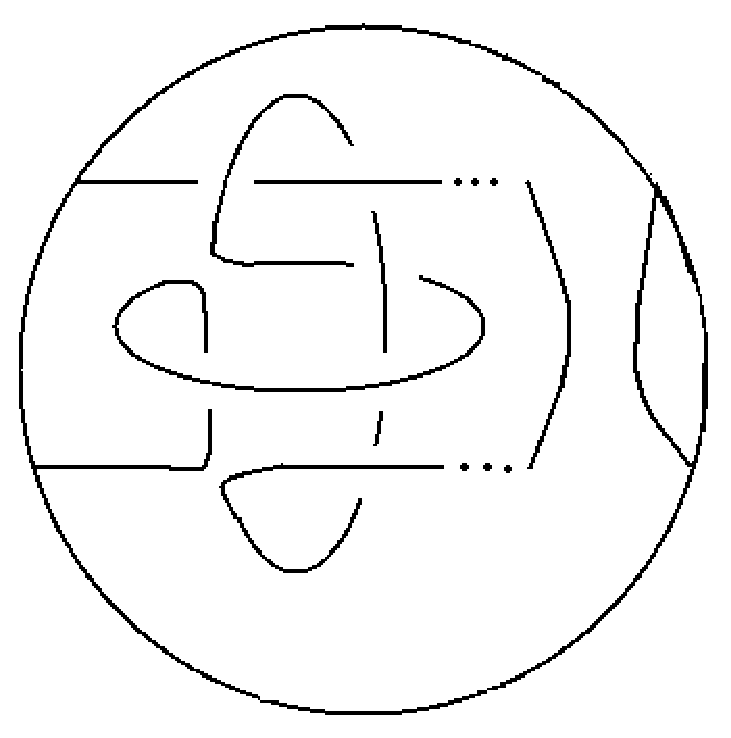}
\includegraphics[scale=0.35]{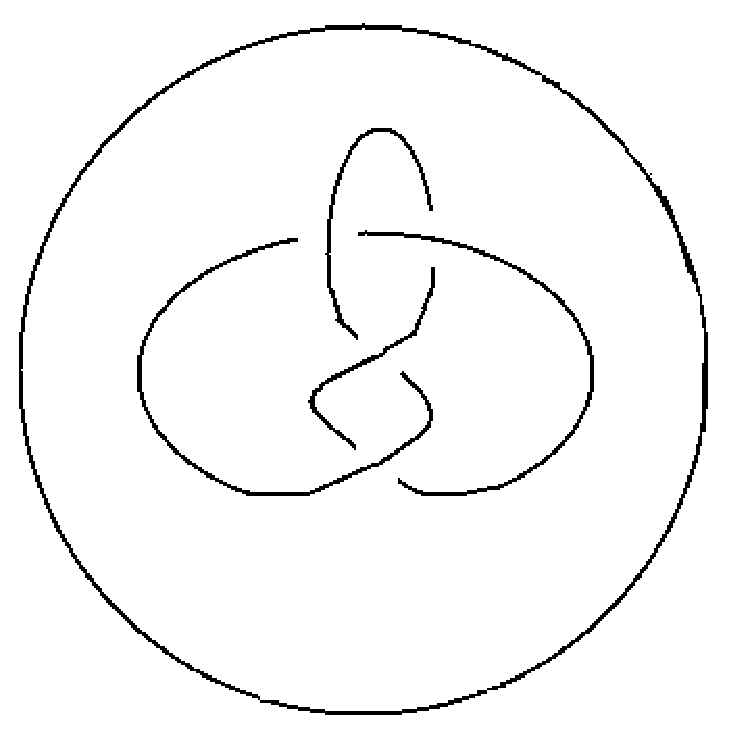}
\caption{Degree 2+2 = Degree 4, Degree 4+2 = Degree 6}
\end{figure}

\clearpage
\section*{Acknowledgments}
I would like to thank Oleg Viro and Tobias Ekholm for valuable ideas and interesting discussions.





\end{document}